\documentclass[12pt, reqno]{amsart}
\usepackage [latin1]{inputenc}
\usepackage{amscd}
\usepackage{epsfig}
\usepackage{amssymb}
\usepackage{amsmath}
\usepackage{amsthm}
\usepackage[T1]{fontenc}
\usepackage{ae,aecompl}

\usepackage {amsfonts} 
\usepackage {latexsym} 
\usepackage {exscale} 
\usepackage{amsmath}

\newcommand{\R} {\ensuremath{\mathbb{R}}}
\newcommand{\N} {\ensuremath{\mathbb{N}}}
\newcommand{\C} {\ensuremath{\mathbb{C}}}
\newcommand{\Z} {\ensuremath{\mathbb{Z}}}

\newcommand{\OO}{\mathcal{O}}
\renewcommand{\o}[1]{\overline{#1}}

\newcommand{\dq}{\overline{\partial}}
\newcommand{\wt}[1]{\widetilde{#1}}
\DeclareMathOperator{\Reg}{Reg}
\DeclareMathOperator{\Sing}{Sing}

\DeclareMathOperator{\Jac}{Jac}

\newtheorem {satz} {Satz} [section]
\newtheorem {lem} [satz] {Lemma}

\newtheorem {defn} [satz] {Definition}

\newtheorem {thm} [satz] {Theorem}

\DeclareMathOperator{\dist}{dist}
\DeclareMathOperator{\supp}{supp}

\renewcommand{\Re}{\mbox{Re }}

\renewcommand{\theta}{\vartheta}

\title[$\dq$ at regular exceptional sets] 
{About the $\dq$-equation at isolated singularities with regular exceptional set}

\author{J. Ruppenthal}

\address{Mathematisches Institut, Universit\"at Bonn, Beringstr. 1, D-53115 Bonn, Germany.}
\email{jean@math.uni-bonn.de}

\date{September 19, 2007}


\begin{document}


\begin{abstract} 
Let $Y$ be a pure dimensional analytic variety in $\C^n$ with an isolated singularity at the origin
such that the exceptional set $X$ of a desingularization of $Y$ is regular.
The main objective of this paper is to present a technique which allows to determine
obstructions to the solvability of the $\dq$ equation in the $L^2$ respectively $L^\infty$ sense
on $Y^*=Y\setminus\{0\}$ in terms of certain cohomology classes on $X$.
More precisely, let $\Omega \subset\subset Y$ be a Stein domain with $0\in \Omega$,
$\Omega^*=\Omega\setminus \{0\}$.
We give a sufficient condition for the solvability of the $\dq$ equation in the $L^2$-sense on $\Omega^*$;
and in the $L^\infty$ sense, if $\Omega$ is in addition strongly pseudoconvex.
If $Y$ is an irreducible cone, we also give some necessary conditions and
obtain optimal H\"older estimates for solutions of the $\dq$ equation.
\end{abstract}

\maketitle

\section{Introduction}

The Cauchy-Riemann differential equations play a decisive role in complex analysis.
Fundamental questions like the Cousin problems or the Levi problem have been solved 
using solutions of the $\dq$-equation with certain estimates
on complex manifolds. But, until now, the $\dq$-theory has not been understood completely
on complex spaces with singularities which occur naturally in complex analysis.
Particularly, the solution of the Cauchy-Riemann equations could be a tool
in the solution of the Levi problem on singular Stein spaces.\\

One difficulty is the question of how to define differential forms
in the presence of singularities. 
There are basically three different approaches:

Firstly, one can use the concept of differential sheaves
that is common in algebraic geometry. 
In \cite{Rp}, it is shown
that in this algebraic setup the $\dq$-equation is not necessarily locally solvable in singular points,
i.e. the Lemma of Dolbeault is not valid.

A second method is to consider forms on some neighborhood
of an embedded variety. In this extrinsic setup, Henkin and Polyakov
obtain the following result: If $X$ is a complete intersection
in the unit ball $B\subset\C^n$, then for each $C^\infty$-form $\alpha$ on $B$ with $\dq \alpha$ vanishing on $X$,
there is a $C^\infty$-form $\beta$ on $B\setminus \Sing X$, such that $\dq \beta=\alpha$
on $X\setminus \Sing X$ (\cite{HePo}).
Counter-examples show that it is not possible to obtain good H\"older-estimates
in this extrinsic setup (cf. \cite{Rp}).\\

On the other hand, one obtains promising results in the following intrinsic setup:
Consider $\dq$-closed forms on the set of regular points of the variety, and estimate solutions 
to the $\dq$-equation on this nonsingular manifold.
Solvability of the Cauchy-Riemann equations in this setup implies solvability
in the extrinsic setup of Henkin and Polyakov (cf. \cite{Rp}).
A central point in complex analysis is the intimate
connection between the geometry of the boundary of certain domains on the one hand,
and the solvability and regularity of the Cauchy-Riemann equations on the other hand.
This is reflected in this intrinsic setup since the singularities
occur as (quite irregular) boundaries.
Aim of this paper is to contribute to the $\dq$-theory on singular Stein spaces
in this intrinsic setup.
As for the $\dq$-equation on manifolds,
we have to distinguish essentially two fields of research:

\vspace{2mm}
{\bf 1.} Solution of the $\dq$-equation with $L^2$-estimates (resp. $L^p$-estimates),

\vspace{2mm}
{\bf 2.} Solution of the $\dq$-equation with H\"older-estimates (resp. $C^{k+\alpha}$-estimates).

\vspace{2mm}
Due to the incompleteness of the metric on singular spaces,
there are several closed extensions of the $\dq$-operator (at least in the $L^2$-case).
In this paper, we will always refer to the maximal extension, i.e. the $\dq$-equation
has to be understood in the sense of distributions.

Whereas the $L^2$-theory on singular Stein spaces
has been treated in a number of 
publications by Forn{\ae}ss, Diederich,
{\O}vrelid and Vassiliadou (\cite{Fo}, \cite{DFV}, \cite{FOV1}, \cite{FOV2}, \cite{OvVa})
during the last ten years,
new promising results in the H\"older-theory were received recently:
In this area, the first exemplary result was given by Forn{\ae}ss and Gavosto in 1998 (\cite{FoGa}),
followed by a variety of examples given by Ruppenthal in 2006 (\cite{Rp2}),
the treatment of rational double points\footnote{Simple (isolated) singularities of hypersurfaces in $\C^3$}
by Acosta and Zeron (\cite{AcZe1,AcZe2}), and more general quotient varieties by Solis and Zeron (\cite{SoZe}).
The last result is an explicit $\dq$-integration formula for weighted homogeneous varieties (\cite{RuZe}),
which we will use in this paper.\\

One method to attack the $\dq$-problem on singular Stein spaces
is to use Hironaka's desingularization in order to blow up the singularities,
and to investigate the $\dq$-equation on the resulting regular complex manifold.
This approach has been pursued only once (in \cite{FOV1}).
Two major difficulties of this procedure are to control
the deformation of the metric at the singularities under desingularization,
and that the resulting space contains positive-dimensional compact complex submanifolds.
The main objective of the present paper is to develop strategies for treating these problems.
This is done in the following situation: Let $Y$ be a pure dimensional analytic set in $\C^n$ with
an isolated singularity at the origin. Then, the exceptional set $X=\pi^{-1}(\{0\})$
of a desingularization $\pi: M \rightarrow Y$ is an exceptional set in the sense of Grauert (\cite{Gr1}).
This implies that $X$ has a weakly negative normal bundle if $X$ is regular,
a fact that we will exploit. Let's explain that more precisely.\\


Let $Y$ be a pure $d$-dimensional analytic variety in $\C^n$,
and $A=Y_{sing}\subset Y$ the analytic subset of singular points of $Y$.
Let $Y^*=Y_{reg}$ be the complex manifold of the regular points of $Y$
provided with the induced metric, so that $Y^*$ is a hermitian
submanifold in $\C^n$, $|\cdot|_Y$ the induced norm on $\Lambda T^* Y^*$
and $dV_Y$ the volume element.
For $U\subset Y$ open, let $U^*:=U\setminus A \subset Y^*$.
Then, for a measurable $(r,q)$-form $\omega$ on $U^*$ and $1\leq p<\infty$,
we set 
\begin{eqnarray*}
\|\omega\|^p_{L^p_{r,q}(U^*)} &:=& \int_{U^*} |\omega|^p dV_X\ ,\\
\|\omega\|_{L^\infty_{r,q}(U^*)} &:=& \mbox{ess} \sup_{\substack{z\in U^*}} |\omega|(z).
\end{eqnarray*}

Note that this definition
coincides with the definition of the concerned
function spaces in almost all treatments of the $\dq$-equation
on singular spaces.
In this paper, we will study solvability and regularity of the
$\dq$-equation
on the set of regular points $Y^*$. 
All differential
equations have to be understood in the sense of distributions.

Our observations are based on Hironaka's desingularization of algebraic varieties
over fields of characteristic zero. Particularly, every reduced complex space
can be desingularized, and every reduced closed complex subspace 
of a complex manifold admits an embedded desingularization.
Precisely, in our case, there exists a complex manifold
$M$ of dimension $d$ and a holomorphic projection
$\pi: M \rightarrow Y$
satisfying the following properties:
\begin{itemize}
\item $X=\pi^{-1}(A)$ is a hypersurface in $M$ with only normal crossings, i.e.
for each point $P\in X$, there are local coordinates $z_1, ..., z_d$, such that
in a neighborhood of $P$, $X$ is the zero set of $h(z)=z_1 \cdots z_k$, where $1\leq k\leq d$.

\item $\pi: M\setminus X \rightarrow Y \setminus A$ is biholomorphic.

\item $\pi$ is a proper mapping.
\end{itemize}

For the proof and more detailed information about Hironaka's desingularization,
we refer to \cite{AHL}, \cite{BiMi} and \cite{Ha}.
This approach to the $\dq$-problem on singular analytic sets has been already
pursued by Forn{\ae}ss, {\O}vrelid and Vassiliadou in their paper \cite{FOV1}:
Let $d_A$ be the distance to $A$ in the ambient $\C^n$.
Then, for $U\subset Y$ open, $U^*=U\setminus A$ and $N\in \N$,
they define weighted $L^2$-norms
\begin{eqnarray*}
\|\omega\|^2_{N,L^2_{p,q}(U^*)} &:=& \int_{U^*} |\omega|^2 d_A^{-N}dV_X.
\end{eqnarray*}
The main result of Forn{\ae}ss, {\O}vrelid and Vassiliadou may be stated as follows:

\begin{thm}\label{thm:FOV1}
Let $Y$ be as above, and $\Omega\subset\subset Y$ an open relatively compact Stein domain in $Y$.
Then for every $N_0\geq 0$, there exists $N\geq 0$
such that if $f$ is a $\dq$-closed $(p,q)$-form on $\Omega^*=\Omega\setminus A$, 
$q>0$, with $\|f\|_{N,L^2_{p,q}(\Omega^*)}<\infty$,
there is $v\in L^{2,loc}_{p,q-1}(\Omega^*)$ solving $\dq v=f$, with $\|v\|_{N_0,L^2_{p,q-1}((\Omega')^*)}<\infty$
for each $\Omega'\subset\subset \Omega$,
and for every $\Omega'\subset\subset \Omega$, there is a solution of this kind satisfying
$\|v\|_{N_0,L^2_{p,q-1}((\Omega')^*)} \leq C \|f\|_{N,L^2_{p,q}(\Omega^*)}$, where $C>0$ 
depends only on $\Omega'$, $N$ and $N_0$.
\end{thm}

In other words, they were able to solve the equation $\dq v=f$ in $\Omega^*$
for a $\dq$-closed $(p,q)$-form $f$ on $\Omega^*$ that vanishes to
sufficiently high order on $A$. Their proof is based on Hironaka's desingularization
and cohomological arguments in the spirit of Grauert \cite{Gr1}.
The constant $N$ appearing in Theorem \ref{thm:FOV1} cannot be controlled,
since it comes into life in a rather abstract manner. It depends on
the exponent in an application
of R\"uckert's Nullstellensatz (\cite{FOV1}, Lemma 2.1)
and on the exponent of some Lojasiewicz inequality (\cite{FOV1}, Lemma 3.2).

Now, the principle idea of the present paper
is that it could be possible to clearify the vanishing assumptions
on $f$ if one takes into account the very special structure
of the embedding of the exceptional set $X=\pi^{-1}(A)$ in $M$.
Including also a careful analysis of the behavior of the norms under the desingularization
$\pi: M\rightarrow X$,
we were able to do that for the $\dq$-equation
if $Y$ is an analytic set with an isolated singularity at the origin
such that the exceptional set $X=\pi^{-1}(0)$ is regular.
In that case we can exploit the fact that $X$ is a compact complex submanifold in $M$
with negative normal bundle. Our strategy is as follows:

Choose a hermitian metric on $M$.
Let $W\subset M\setminus X$ be open. 
We may assume that $\pi: W \rightarrow \pi(W)\subset Y\setminus A$
preserves orientation. Then the determinant of the real Jacobian 
$(\det \Jac_\R \pi)$ is a
well-defined positive function on $W$, and we have
$$\int_{\pi(W)} g\ dV_Y = \int_{W} g\circ\pi\ (\det\Jac_\R \pi)  dV_M$$
for measurable functions $g$ on $\pi(W)$. 
Note that $\pi^* dV_Y = (\det\Jac_\R \pi)\ dV_M$.
Here we make the following crucial observation:
If $W$ is chosen appropriately, then there exists a holomorphic function $J$ on $W$,
such that $(\det\Jac_\R \pi)=|J|^2$. 
This will be shown in Lemma \ref{lem:detjac}.

Now, let $\Omega\subset\subset Y$ be a Stein domain such that $0\in \Omega$,
and $\omega\in L^2_{0,1}(\Omega^*)$.
We consider the pull-back $\pi^* \omega$ on $\Omega'\setminus X$, where $\Omega'=\pi^{-1}(\Omega)$.
Unfortunately, this form is neither $L^2$ nor $\dq$-closed.
But, by the use of Lemma \ref{lem:detjac},
we observe that the locally defined form $J\cdot\pi^* \omega$ is square integrable.
Here a generalization of Riemann's Extension Theorem
comes into play, which is an important tool in the study of the $\dq$-equation on 
analytic varieties: \vspace{3mm}\newline
{\bf Theorem \ref{thm:extension}.}
{\it Let $A$ be an analytic subset of an open set $D\subset\C^n$
of codimension $k>0$.
Furthermore, let $f\in L^{(2k)/(2k-1)}_{(0,q),loc}(D)$
and $g\in L^1_{(0,q+1),loc}(D)$ be differential forms
such that $\dq f = g$ on $D\setminus A.$
Then $\dq f= g$ on the whole set $D$.}

\vspace{3mm}
We conclude that the locally defined product $J\cdot\pi^* \omega$ is $\dq$-closed in the sense of distributions
(if we extend it trivially across $X$). 
This inspires the following construction:
Let $X$ be an analytic subset in $M$ and $\mathcal{I}$ the ideal sheaf of $X$ in $M$.
Then $\mathcal{I}^k \OO_M$ is a subsheaf of the sheaf of germs of meromorphic
functions on $M$ for each $k\in\Z$.

In order to give a fine resolution for the $\mathcal{I}^k\OO_M$,
we define 
the presheaves
$$\mathcal{L}_{0,q}(U):=\{f\in L^{2,loc}_{0,q}(U): \dq f\in L^{2,loc}_{0,q+1}(U)\},\ \ U\subset M \mbox{ open},$$
which are already sheaves with the natural restriction maps.
For $k\in \Z$, we consider $\mathcal{I}^k \mathcal{L}_{0,q}$
as subsheaves of the sheaf of germs of differential forms with measurable coefficients.
Now, we define a weighted $\dq$-operator on $\mathcal{I}^k \mathcal{L}_{0,q}$.
Let $f \in (\mathcal{I}^k \mathcal{L}_{0,q})_z$. Then $f$ can be written locally as $f = h^k f_0$,
where $h\in (\OO_M)_z$ generates $\mathcal{I}_z$ and $f_0\in (\mathcal{L}_{0,q})_z$. Let
$$\dq_k f := h^k \dq f_0 = h^k \dq ( h^{-k} f).$$
For $k\geq 0$, this operator is just the usual $\dq$-operator.
We obtain the fine resolution
$$0 \rightarrow \mathcal{I}^k\OO_M \rightarrow
\mathcal{I}^k \mathcal{L}_{0,0} \xrightarrow{\ \dq_k\ }
\mathcal{I}^k \mathcal{L}_{0,1} \xrightarrow{\ \dq_k\ }
\cdots \xrightarrow{\ \dq_k\ }
\mathcal{I}^k \mathcal{L}_{0,d} \rightarrow 0 .$$
By the abstract Theorem of de Rham, this implies
$$H^q(U,\mathcal{I}^k\OO_M) \cong \frac{\mbox{ker }
(\dq_k: \mathcal{I}^k\mathcal{L}_{0,q}(U) \rightarrow \mathcal{I}^k\mathcal{L}_{0,q+1}(U))}
{\mbox{Im }
(\dq_k: \mathcal{I}^k\mathcal{L}_{0,q-1}(U) \rightarrow \mathcal{I}^k\mathcal{L}_{0,q}(U))}$$
for open sets $U\subset M$. Let us return to $\pi^*\omega$.
If $X$ is irreducible, then all the locally defined holomorphic functions $J$
vanish of a fixed order $k_0>0$ along $X$. Hence, 
$$\pi^*\omega\in \mathcal{I}^{-k_0}\mathcal{L}_{0,1}(\Omega')\ \ \mbox{ and }\ \ \dq_{-k_0} \pi^*\omega=0.$$
It is now interesting to take a closer look at $H^1(V,\mathcal{I}^{-k_0}\OO_M)$
for a well-chosen neighborhood $V$ of $X$ in $M$.
Here, let us assume that $X$ is regular. Then, by the work of Grauert (see \cite{Gr1}),
we know that the normal bundle $N$ of $X$ in $M$ is negative.
This implies that there is a neighborhood of $X$ in $M$ which is biholomorphic to a neighborhood of the
zero section in the normal bundle (Theorem \ref{thm:embedding}).
So, there is a
neighborhood $V$ of $X$ in $\Omega'\subset\subset M$ such that
$$H^q(V,\mathcal{I}^{-k_0}\OO_M) \cong H^q(N,\mathcal{I}^{-k_0}\OO_{N}),$$
where, on the right hand side, $\mathcal{I}$ denotes the ideal sheaf of the zero section in $N$.
Using Grauert's vanishing results for the cohomology of coherent analytic sheaves
with values in positive line bundles and an expansion for cohomology classes on holomorphic line bundles (Theorem \ref{thm:cohom}), 
we were able to show that
\begin{eqnarray}\label{eq:obstructions}
H^q(N,\mathcal{I}^{-k_0}\OO_{N}) \cong \bigoplus_{\mu\geq -k_0} H^q(X,\OO(N^{-\mu})),
\end{eqnarray}
where $\OO(N^{-\mu})$ is the sheaf of germs of holomorphic sections in $N^{-\mu}$.
This approach allows to compute obstructions to the solvability of the $\dq$-equation.
We will do that in more detail for homogeneous varieties with an isolated singularity.
For the moment, assume that the right hand side of \eqref{eq:obstructions}
vanishes. Then, we deduce that there is a solution $\eta'\in \mathcal{I}^{-k_0}\mathcal{L}_{0,0}(V)$
such that $\dq_{-k_0}\eta' = \pi^*\omega$ on $V$.
But then $\eta:=(\pi^{-1})^* \eta' \in L^2(U\setminus\{0\})$ for $U\subset\subset \pi(V)$
and $\dq \eta=\omega$ on $\pi(V)\setminus\{0\}$.

This solution can be extended to the whole
set $\Omega^*$ by the use of H\"ormander's $L^2$-theory (\cite{Hoe}).
We obtain the following result:

\begin{thm}\label{thm:l2}
Let $Y$ be a pure dimensional analytic variety in $\C^n$
with an isolated singularity at the origin
such that the exceptional set $X$ of a desingularization of $Y$ is regular, and denote by $N$ the normal bundle
of $X$.
Then there exists a natural number $k_0=k_0(Y)>0$ such that the following is true:
Assume that $H^1(X,\OO(N^\mu))=0$ for all $\mu\leq k_0$ and that
$\Omega\subset\subset Y$ is a Stein domain with $0\in\Omega$, $\Omega^*=\Omega\setminus\{0\}$.
Then there exists a constant $C_\Omega >0$ such that:
If $\omega\in L^2_{0,1}(\Omega^*)$ with $\dq\omega=0$,
then there exists a solution $\eta\in L^2(\Omega^*)$ such that $\dq \eta=\omega$
and
$$\|\eta\|_{L^2(\Omega*)} \leq C_\Omega\ \|\omega\|_{L^2_{0,1}(\Omega^*)}.$$
\end{thm}


Here, one should mention the following Cohomology Extension Theorem:

\begin{thm}{(\bf Scheja \cite{Sch1,Sch2})}\label{thm:scheja}
Let $Y$ be a closed pure dimensional analytic subset in $\C^n$ which is locally a complete
intersection, and $A$ a closed pure dimensional analytic subset of $Y$.
Then, the natural restriction mapping
$$H^q(Y,\OO_Y) \rightarrow H^q(Y\setminus A,\OO_{Y\setminus A})$$
is bijective for all $0\leq q\leq \dim Y-\dim A-2$.
\end{thm}

If $Y$ is a complete intersection with an isolated singularity at the origin,
and $\Omega\subset\subset Y$ Stein with $0\in\Omega$, then it can be deduced that there is a bounded
$L^2$-solution operator for the $\dq$ equation on $\Omega\setminus\{0\}$ for $(p,q)$-forms
if $q>0$ and $1\leq p + q \leq \dim Y-2$. 
This was shown by Forn{\ae}ss, {\O}vrelid and Vassiliadou in \cite{FOV2}.
One must emphasize that Theorem \ref{thm:l2} is based on completely different techniques.

Now, let $\Omega\subset\subset Y$ be strongly pseudoconvex and $\omega\in L^\infty_{0,1}(\Omega^*)$.
This is a much more comfortable situation. By use of a regularization procedure
(we solve the $\dq$-equation two times locally with certain estimates),
we can assume that 
$$\pi^* \omega \in \mathcal{I}^1\mathcal{L}_{0,1}(\Omega')\ \ \mbox{ and }\ \ \dq_1\pi^*\omega=\dq\pi^*\omega=0.$$
As before, there is the neighborhood $V$ of $X$ such that
$$H^q(V,\mathcal{I}^1\OO_M) \cong H^q(N,\mathcal{I}^1 \OO_{N}).$$
So, in this case, it is enough to assume that 
$$H^1(X,\OO(N^{-\mu}))=0$$
for all $\mu \geq 1$.
In that case, we conclude that there is a continuous
solution $\eta'\in \mathcal{I}^1 \mathcal{L}_{0,0}(V)$, $\dq_1 \eta'=\dq \eta'=\pi^*\omega$,
which is identically zero on $X$.
This implies that we have $\eta:=(\pi^{-1})^*\eta' \in C^0(\pi(V))$ such that $\dq \eta=\omega$
on $\pi(V)\setminus\{0\}$. If we restrict $\eta$ to a strongly pseudoconvex neighborhood $U\subset\subset \pi(V)$
of the origin, then
$\eta$ can be extended to the whole set $\o{\Omega}$ by the use of Grauert's bump method.
This leads to the following Theorem:

\begin{thm}\label{thm:c0}
Let $Y$ be a pure dimensional analytic set in $\C^n$
with an isolated singularity at the origin
such that the exceptional set $X$ of a desingularization of $Y$ is regular and $H^1(X,\OO(N^{-\mu}))=0$ for all
$\mu\geq 1$, where $N$ is the normal bundle of $X$.
If $\Omega\subset\subset Y$ is a strongly pseudoconvex subset with $0\in\Omega$ and $\Omega^*=\Omega\setminus\{0\}$,
then there exists a constant $C_\Omega >0$ such that the following is true:
If $\omega\in L^\infty_{0,1}(\Omega^*)$ with $\dq\omega=0$,
then there exists 
$\eta\in C^0(\o{\Omega})$ such that $\dq \eta=\omega$ on $\Omega^*$
and
$$\|\eta\|_{C^0(\o{\Omega})} \leq C_\Omega\ \|\omega\|_{L^\infty_{0,1}(\Omega^*)}.$$
\end{thm}

From the proof, it will be clear that the solution $\eta$ has much better properties
than just being continuous. We will work that out in the case where $Y$ is a cone
with an isolated singularity at the origin. In this situation,
a desingularization of $Y$ is given by a single blow up, and the exceptional set $X$ is regular.
Then, we can use a $\dq$-integration formula for weighted homogeneous varieties (see \cite{RuZe})
in order to estimate the solution at the origin, where it proves to be H\"older $\alpha$-continuous
for each $0<\alpha<1$. Taking into account the H\"older $1/2$-continuity at the strongly pseudoconvex
boundary of the domain $\Omega$, it follows:

\begin{thm}\label{thm:hoelder}
Let $Y$ be an irreducible cone in $\C^n$ with an isolated singularity at the origin
such that $H^1(X,\OO(N^{-\mu}))=0$ for all $\mu\geq 1$, where $X$ is the exceptional set
of the blow up at the origin and $N$ its normal bundle.
Moreover, let
$\Omega\subset\subset Y$ be strongly pseudoconvex with $0\in\Omega$ and $\Omega^*=\Omega\setminus\{0\}$.
Then there exists a constant $C_\Omega' >0$ such that the following is true:
If $\omega\in L^\infty_{0,1}(\Omega^*)$ with $\dq\omega=0$,
then there exists a solution $\eta\in C^{1/2}(\o{\Omega})$ with $\dq \eta=\omega$ on $\Omega^*$
and
$$\|\eta\|_{C^{1/2}(\o{\Omega})} \leq C_\Omega'\ \|\omega\|_{L^\infty_{0,1}(\Omega^*)}.$$
\end{thm}

Here, the H\"older-norm is defined as follows:
Let $z,w\in Y$.
Then we define $\dist_Y(z,w)$
as the infimum of the length of all curves connecting $z$ and $w$ in $Y$
that are piecewise smooth. It is clear that such curves exist
in this situation.
The length is measured as the length of the curve in 
the ambient space $\C^n$, which is the same as the length in $Y^*$
since $Y^*$ carries the induced norm. For $0<\alpha<1$, $U\subset Y$ and $f$ a measurable function on $U$, 
let
$$\|f\|_{C^\alpha(U)} := \|f\|_{L^\infty(U)} + 
\sup_{\substack{z,w\in U\\ z\neq w}} \frac{|f(z)-f(w)|}{\dist_Y(z,w)^\alpha}.$$

The paper is organized as follows:
In the next four chapters, we will present the main tools of our work.
They are developed partially in more general versions than needed in this paper.
Section 6 is dedicated to the proof of Theorem \ref{thm:l2}, while we will
prove Theorem \ref{thm:c0} in section 7.
In section 8, we restrict our attention to irreducible homogeneous varieties with an isolated 
singularity. We will prove Theorem \ref{thm:hoelder} and determine some necessary conditions
for $L^\infty$-solvability of the $\dq$-equation in that setting. Then,
we compute some obstructions explicitly for certain examples in section 9,
and conclude this paper by a short remark on the cohomology of a desingularization.

\newpage
\section{Behavior of the $L^2$-Norm under Desingularization}\label{sec:detjac}

Let $Y$ be a closed analytic set of pure dimension $d\geq 2$ in $\C^n$ and $\pi: M\rightarrow Y$
a desingularization as in the introduction.
Furthermore, let $A = Y_{sing}$ be the singular set in $Y$ and $X=\pi^{-1}(A)$.
Then $X$ has only normal crossings. For $a\in\C^d$ and $r>0$, let
$$P_r(a)=\{z\in\C^d: |a_j-z_j| < r \mbox{ for } j=1, ..., d\}.$$
We call $\Psi: P_r(0) \rightarrow M$
a {\bf nice holomorphic polydisc} (with respect to a desingularization $\pi: M \rightarrow Y$),
if $\Psi$ is biholomorphic and
$$\Psi^{-1}(X) = \{z\in P_r(0): z_1\cdots z_k=0\}$$
for some $1\leq k\leq d$. In this section, we will study $\det \Jac_\R \pi$
on compact subsets of $\Psi(P_r(0))$. Since $\Psi$ is a biholomorphism, we may as well
investigate 
$$\Pi:=\pi\circ\Psi:\ P_r(0)\subset\C^d \rightarrow Y.$$

Let $\wt{A}=\overline{\Pi^{-1}(A)}$ and $\Pi':=\Pi|_{P_r(0)\setminus \wt{A}}: P_r(0)\setminus\wt{A}\rightarrow Y$,
which is a biholomorphism onto its image.
We need an extension of $\Pi'$ to a biholomorphic map on a space with vanishing
first singular cohomology group.
For this, let's construct an artificial codomain:
Let $0< r' <r$ and 
\begin{eqnarray*}
Y_1 &:=& \C^d\setminus \overline{\big( P_{r'}(0) \cup \wt{A}\big)}\ \subset \C^d,\\
Y_2 &:=& \Pi'(P_r(0) \setminus \wt{A}) \ \subset Y.
\end{eqnarray*}
We define the complex manifold
$$\wt{Y} := Y_1 \dot\cup Y_2 / \sim\ \ ,$$
where for $y_1\in Y_1$, $y_2\in Y_2$ we say $y_1\sim y_2$ if $y_2=\Pi'(y_1)$.
Since $\Pi'$ is a biholomorphism onto $Y_2$, we obtain a complex manifold
$\wt{Y}$, which is covered by the two open sets $Y_1$ and $Y_2$. Let $(\chi_1, \chi_2)$
be a partition of unity subordinate to that open cover, $\chi_2=1-\chi_1$.
Let $h_1$ be the standard hermitian metric on $Y_1 \subset \C^d$ and
$h_2$ the restriction to $Y_2 \subset Y$ of the hermitian metric on $Y$, which was
induced by the embedding $Y\subset \C^n$. Then $\wt{Y}$ becomes a hermitian manifold
with hermitian metric $h:=\chi_1 h_1 + \chi_2 h_2$.
Now, we define the biholomorphic extension
$$\wt{\Pi}: \C^d \setminus \wt{A} \rightarrow \wt{Y},\ \ z\mapsto 
\left\{\begin{array}{ll}
z &, \mbox{if } z\in Y_1,\\
\Pi'(z) &, \mbox{if } z\in P_r(0).
\end{array}\right.$$

Now, cover $\C^d \setminus \wt{A}$ with (countably many) balls $B_1, B_2, ...$, such that there are orthonormal
coordinates in $\wt{\Pi}(B_j)$ and finite intersections of such balls are Stein.
Let $G_j: B_j\rightarrow GL_d(\C)$ be the complex Jacobian of $\wt{\Pi}$
in these coordinates and $F_j:=\det G_j$, which is a holomorphic function on $B_j$.
Here we do not write $(\det \Jac_\C)$, because, unlike the determinant of the real Jacobian,
this expression is not invariant under change of coordinates.

But we know that
\begin{eqnarray*}
\det \Jac_\R \wt{\Pi} (z)= |F_j(z)|^2
\end{eqnarray*}
for $z\in B_j$ (see e.g. \cite{Ra}, Lemma I.2.1). 
So, under a change of coordinates,
we deduce that $|F_j(z)|=|F_k(z)|$ for $z\in B_j\cap B_k$. The quotients
$c_{jk}:=F_j/F_k$
are holomorphic functions on $B_j\cap B_k$ with $|c_{jk}|\equiv 1$,
hence locally constant with values in $S^1\subset\C$. It is clear that the
$\{c_{jk}\}$ form a multiplicative $1$-cocycle with values in the 
constant sheaf $\overline{\C}$. We would prefer an additive cocycle.
Hence, choose locally constant functions $a_{jk}$ with values in $\R$
such that $c_{jk}=e^{i\cdot a_{jk}}$. But we also require that
$a_{jk}+a_{kl}+a_{lj}=0$. We may assume that this is satisfied
if the underlying space $\C^d \setminus \wt{A}$ has the Oka property,
namely $H^2(\C^d \setminus \wt{A},\Z)=0$ (cf. \cite{Ra}, VI.5.4 for more detailed
information). But $\C^d\setminus \wt{A}$ is homotopy equivalent to the sphere 
$S^{2d-1}$ and therefore the Oka property is satisfied since $d\geq 2$.
So, we assume that $\{a_{jk}\}$ is an additive $1$-cocycle with values in the
constant sheaf $\overline{\R}$. Now, recall that $\mathcal{B}=\{B_j\}$ is an
acyclic covering of $\C^d\setminus \wt{A}$.

Using again the fact that this space
is homotopic equivalent to a sphere $S^{2d-1}$, we observe that 
$\check{H}^1(\mathcal{B},\overline{\R}) \cong H^1(\C^d\setminus \wt{A}, \R)=0$, too.
But vanishing of the first \v{C}ech-Cohomology group implies that the $1$-cocycle
$\{a_{jk}\}$ is solvable. Thus, there exist locally constant functions $a_j$
on $B_j$ with values in $\R$ such that $a_{jk}=a_j-a_k$.
Consequently, $c_j:=e^{i\cdot a_j}$ are holomorphic functions on $B_j$
such that $c_{jk}=c_j/c_k$ and $|c_j|\equiv 1$. Now 
\begin{eqnarray*}
J(z) := F_j(z) c_j^{-1}(z)\ \ ,\ \mbox{ for } z\in B_j,
\end{eqnarray*}
is a globally well-defined holomorphic function on $\C^d\setminus\wt{A}$, for
on $B_j\cap B_k$:
$$F_j \cdot c_j^{-1} = F_k \cdot c_{jk}\cdot c_j^{-1} 
= F_k\cdot c_{j}\cdot c_{k}^{-1}\cdot c_j^{-1} = F_k\cdot c_k^{-1}.$$
Moreover, for $z\in P_{r'}(0)\setminus \wt{A}$:
$$|J(z)|^2=|F_j(z)|^2=\det \Jac_\R \wt{\Pi}(z) = \det\Jac_\R \Pi(z).$$
But as $J$ is a holomorphic function on $\C^d\setminus\wt{A}$,
it extends holomorphically across $\wt{A}$  by Hartog's extension theorem.
Continuity implies that $|J(z)|^2 = \det\Jac_\R \Pi(z)$ for all $z\in P_{r'}(0)$.
We summarize (the case $d=1$ is almost trivial):

\begin{lem}\label{lem:detjac}
Let $\pi: M\rightarrow Y$ be a desingularization as in the introduction and
$\Psi: P_r(0)\subset \C^d \rightarrow M$ a nice holomorphic polydisc
(with respect to $\pi$), $\Pi:=\pi\circ\Psi$. Let $0<r'<r$.
Then there exists a holomorphic function $J$ on $P_{r'}(0)\subset\subset P_r(0)$
such that
$$\det \Jac_\R \Pi(z) = |J(z)|^2\ \ \mbox{ for all } z\in P_{r'}(0).$$
If $W$ is an open subset of $\Psi^{-1}(Y_{reg})$, it follows that
$$\int_{\Pi(W)} g\ dV_Y = \int_W g\circ\Psi\ |J|^2 dV_{\C^d}$$
for measurable functions $g$ on $\Pi(W)$. Thus:
$g\in L^2(\Pi(W))$ $\Leftrightarrow$ $J\cdot \Pi^* g\in L^2(W)$.
\end{lem}

Note that there is no hope in constructing a holomorphic function
with that properties globally on $M$ since that would give a holomorphic function $h$ 
on $Y$ vanishing exactly in the singular set.
This would have a least two strong consequences. On the one hand $Y_{reg}$ would be holomorphically
convex, hence a Stein space itself. On the other hand, $Y$ could not be normal,
since, in that case, $h^{-1}$ would extend holomorphically across the (at least $2$-codimensional)
singular set $Y_{sing}$ by the second Extension Theorem of Riemann.

\newpage
\section{Extension of the $\dq$-Equation across Analytic Sets}

Let us state Riemann's Extension Theorems in the following manner:

\begin{thm}\label{thm:riemann}
Let $A$ be an analytic subset of an open set $D\subset\C^n$
of codimension $k>0$, and $f\in L^p_{loc}(D)$ for some $p\geq 1$ with
$\dq f = 0$ on $D\setminus A$. Then
$\dq f = 0$ on the whole set $D$ provided $p\geq 2$ or $k>1$,
and in that case $f\in\mathcal{O}(D)$.
\end{thm}

All differential equations have to be understood in the sense of distributions.
A proof of Theorem \ref{thm:riemann} can be found in \cite{Rp2}.
But what happens if the right hand side of $\dq f=0$
is replaced by some $g\in L^1_{(0,1), loc}(D)$ which is not necessarily zero?
The objective of this section is to show the following
Extension Theorem,
which is an important tool in the study of the $\dq$-equation on analytic varieties:

\begin{thm}\label{thm:extension}
Let $A$ be an analytic subset of an open set $D\subset\C^n$
of codimension $k>0$.
Furthermore, let $f\in L^{(2k)/(2k-1)}_{(0,q),loc}(D)$
and $g\in L^1_{(0,q+1),loc}(D)$ be differential forms
such that $\dq f = g \ \ \mbox{ on } D\setminus A$.
Then $\dq f= g$ on the whole set $D$.
\end{thm}

Note that the condition on $f$ cannot be relaxed furthermore.
To see this, consider the Bochner-Martinelli-kernel
$$K_0(\zeta,z) = -\frac{(n-2)!}{2\pi^n} *\partial_\zeta \|\zeta-z\|^{2-2n},$$
which is a $\dq_\zeta$-closed $(n,n-1)$-form on $\C^n\setminus\{z\}$ for $z\in\C^n$ fixed.
Moreover, $\left|K_0(\zeta,z)\right| \lesssim \|\zeta-z\|^{1-2n}$
implies that
$K_0(\cdot,z) \in L^p_{(n,n-1),loc}(\C^n)$
for $1\leq p<\frac{2n}{2n-1}$, but not for $p=\frac{2n}{2n-1}$.
Now, the equation $\dq_\zeta K_0(\cdot,z)=0$ cannot be true on the whole set $\C^n$,
for, in that case the Bochner-Martinelli representation formula would not be true.
Hence, the assumption $f\in L^{\frac{2k}{2k-1}}$ in Theorem \ref{thm:extension} cannot be dropped,
at least not in codimension $k=n$.
Now to the proof of Theorem \ref{thm:extension}:

\begin{proof}
Assume first of all that $A$ is a complex submanifold in $D$ of dimension $l=n-k <n$.
Since we have to show a local statement which is invariant under biholomorphic change 
of coordinates, we may assume that $A= D \cap M$ for
$$M:= \{z=(z',z'')\in \C^l\times \C^k=\C^n: z''=0\},$$
and that $D$ is bounded. For $r>0$, define
$$U(r):=\{z\in\C^n: \dist(z,M)=\|z''\|<r\}.$$
Choose a smooth cut-off function $\chi\in C^\infty_{cpt}(\R)$ with $|\chi|\leq 1$,
$\chi(t)=1$ if $|t|\leq 1/2$, $\chi(t)=0$ if $|t|\geq 2/3$, and $|\chi'| \leq 8$.
Now, let
$$\chi_r(z):=\chi(\frac{\dist(z,M)}{r}).$$
Then $\chi_r\equiv 1$ on $U(r/2)$ and $\supp \chi_r \subset U(3r/4)$.
$\chi_r$ is smooth
and we have:
$$\|\nabla \chi_r\|\leq |\chi'| / r \leq 8/r.$$
Since $D$ is bounded, there is $R>0$ such that $D\subset B_R(0)$. It follows
\begin{eqnarray*}
\int_{D}\|\nabla \chi_r\|^p dV_{\C^n}&\leq& 8^p \int_{D\cap U(r)} r^{-p} dV_{\C^n}
\leq 8^p (2R)^{2l} r^{-p} \int_{\{x''\in\C^k:\|x''\|<r\}} dV_{\C^k},
\end{eqnarray*}
and we conclude:
\begin{eqnarray}\label{eq:dchi}
\|\nabla\chi_r\|^p_{L^p(D)} \leq C(R,l,p) r^{2k-p}.
\end{eqnarray}
What we have to show is that
\begin{eqnarray}\label{eq:dqe2}
\int_D f\wedge \dq \phi =(-1)^{q+1} \int_D g\wedge\phi
\end{eqnarray}
for all smooth $(n,n-q-1)$-forms $\phi$ with compact support in $D$.
By assumption, $\dq f=g$ on $D\setminus A$. That leads to:
\begin{eqnarray*}
\int_D f\wedge\dq\phi &=& \int_D f\wedge\chi_r \dq\phi +\int_{D\setminus M} f \wedge(1-\chi_r) \dq\phi\\
&=& \int_D f \wedge\chi_r \dq\phi + \int_{D\setminus M} f \wedge\dq[ (1-\chi_r) \phi]-
\int_{D\setminus M} f\wedge\dq(1-\chi_r)\wedge \phi\\
&=& \int_D f \wedge\chi_r \dq\phi + (-1)^{q+1} \int_{D\setminus M} g \wedge(1-\chi_r) \phi + 
\int_{D\setminus M} f \wedge\dq\chi_r\wedge\phi.
\end{eqnarray*}
Let's consider
$$\int_D f\wedge\chi_r\dq\phi\ \  \mbox{ and }\ \  \int_{D\setminus M} g \wedge\chi_r\phi.$$
Since $|\chi_r|\leq 1$, we have
\begin{eqnarray*}
|f\wedge\chi_r\dq\phi| &\leq& |f\wedge\dq\phi| \in L^1,\\
|g\wedge\chi_r\phi| &\leq& |g\wedge\phi| \in L^1,
\end{eqnarray*}
and we know that $f\wedge\chi_r\dq\phi \rightarrow 0$ and
$g\wedge\chi_r\phi \rightarrow 0$ pointwise if $r\rightarrow 0$.
Hence, Lebesgue's Theorem on dominated convergence gives:
\begin{eqnarray*}
\lim_{r\rightarrow 0} \int_D f \wedge\chi_r \dq\phi &=& 0,\\
\lim_{r\rightarrow 0} \int_{D\setminus M} g \wedge(1-\chi_r) \phi &=&  \int_{D\setminus M} g\wedge\phi
=\int_D g \wedge\phi.
\end{eqnarray*}
To prove \eqref{eq:dqe2}, only
$$\lim_{r\rightarrow 0}\int_{D\setminus M} f\wedge\dq\chi_r\wedge\phi=0$$
remains to show. Let $p:=2k$ and $q:=2k/(2k-1)$ (satisfying $1/p+1/q=1$).
Using \eqref{eq:dchi} and the H\"older-inequality, we get
\begin{eqnarray*}
\lim_{r\rightarrow 0} \|f\wedge\dq\chi_r\wedge \phi\|_{L^1(D)}
&=& \lim_{r\rightarrow 0} \|f\wedge\dq\chi_r\wedge \phi\|_{L^1(U(r))}\\
&\leq& \lim_{r\rightarrow 0} \|f \wedge\phi\|_{L^q(U(r))}\|\dq\chi_r\|_{L^p(D)}\\
&\leq& C(R,l,p)^{1/p} \lim_{r\rightarrow 0} \|f \wedge\phi\|_{L^q(U(r))}.
\end{eqnarray*}
Since $f\in L^q$, we conclude
$$\lim_{r\rightarrow 0} \|f \wedge\phi\|_{L^q(U(r))} = 0$$
(cf. for instance \cite{Alt}, Lemma A 1.16), and the Theorem is proved in case $A$ is a complex submanifold.

Now, let $A$ be a closed analytic subset in $D$ of codimension $k=n-l>0$.
Let $A_l := \Reg(A)$ be the complex manifold of the regular points of $A$,
which is a complex submanifold in $D$ of dimension $\leq l$,
and let
$S_l := \Sing(A) = A\setminus \Reg(A)$ be the singular set,
which is a closed analytic subset of dimension $\leq l-1$.
Now, let $A_{l-1} := \Reg(S_l)$ and $S_{l-1} := \Sing(S_l)$.
Going on inductively, we define a disjoint union of complex manifolds $A_j$
of dimension $\leq j \leq l$, satisfying
$\bigcup A_j=A$.
For $j=0, ..., l$, let $M_j$ be the union of all complex manifolds of dimension $j$
which are contained in $A_k$ for all $k\geq j$.
Then the $M_j$ are complex manifolds of dimension $j$ and $A=\bigcup M_j$.
Now, the statement of the theorem for manifolds shows that $\dq f=g$ 
on $(D\setminus A)\cup M_l$.
Another application gives $\dq f=g$ on
$(D\setminus A)\cup M_l\cup M_{l-1}$, and inductively, we conclude that
the statement of the Theorem is true for analytic sets.
\end{proof}

Clearly, Theorem \ref{thm:extension} is also true if $f$ is an $(r,q)$-form,
and it generalizes to other first order differential operators.
For the sake of completeness, we will now consider the situation
where $A$ is a regular real hypersurface. Letting $k\rightarrow 1/2$ in the 
statement of Theorem \ref{thm:extension} would lead to the conjecture
that $f\in L^\infty$ could be sufficient. But as a real hypersurface
separates $D$ into disjoint parts, it is seen easily that we need a little more:

\begin{thm}\label{thm:extension2}
Let $M$ be a regular real hypersurface in an open set $D\subset\C^n$.
Furthermore, let $f\in C^0_{0,q}(D)$
and $g\in L^1_{(0,q+1),loc}(D)$ be differential forms
such that $\dq f = g$ on $D\setminus M$.
Then $\dq f= g$ on the whole set $D$.
\end{thm}

\begin{proof}
We have to modify the proof of Theorem \ref{thm:extension} slightly.
Here, we have to consider the situation
\begin{eqnarray*}
M &:=& \{z=(z',z'')\in \C^{n-1}\times \C^1: x=\Re z''=0\},\\
U(r)& :=& \{z\in\C^n: \dist(z,M)=\|x\|<r\},
\end{eqnarray*} 
and choose $\chi_r$ as before. 
The only problem is to show that
\begin{eqnarray}\label{eq:dqe3}
\lim_{r\rightarrow 0}\int_{D\setminus M} f\wedge\dq\chi_r\wedge\phi=0
\end{eqnarray}
is still satisfied. But we calculate:
$$\frac{\partial \chi_r}{\partial \o{z}}(x)
= \frac{1}{2} \frac{\partial \chi_r}{\partial x}(x)
= -\frac{1}{2} \frac{\partial \chi_r}{\partial x}(-x)
= -\frac{\partial \chi_r}{\partial \o{z}}(-x).$$
Since the coefficients of $f$ and $\phi$ are continuous,
this implies \eqref{eq:dqe3}. In fact:
\begin{eqnarray*}
\lim_{r\rightarrow 0}\int_{\R^-} g(t) \frac{\partial \chi_r}{\partial\o{z}}(t) dt &=& \frac{1}{2}g(0),\\
\lim_{r\rightarrow 0}\int_{\R^+} g(t) \frac{\partial \chi_r}{\partial\o{z}}(t) dt &=& - \frac{1}{2}g(0),
\end{eqnarray*}
if $g$ is a continuous function, 
for this is nothing else but the convolution with a Dirac-sequence.
\end{proof}

We conclude this section with a nice little application of Theorem \ref{thm:extension2}.
It can be used for the construction of composed $\dq$-solution formulas. We consider
the complex projective space $\C\mathbb{P}^1$ with homogeneous coordinates $[w:z]$,
which is covered by the two charts
\begin{eqnarray*}
\Psi_a:  \C \rightarrow \C\mathbb{P}^1,\ \ \ a &\mapsto& [1:a]\ ,\\
\Psi_b:  \C \rightarrow \C\mathbb{P}^1,\ \ \ b &\mapsto& [b:1]\ .
\end{eqnarray*}

\vspace{2mm}
If $g$ is a $(0,1)$-form on $\C\mathbb{P}^1$, let $\Psi_a^* g = g_a d\o{a}$ and $\Psi_b^* g=g_b d\o{b}$.
For $1\leq p \leq \infty$, we define
$$\|g\|_{L^p_{0,1}(\C\mathbb{P}^1)} = \|g_a\|_{L^p(\Delta_1(0))} + \|g_b\|_{L^p(\Delta_1(0))}.$$
Other definitions lead to equivalent norms since $\C\mathbb{P}^1$ is a compact manifold.
Now, fix $2<p<\infty$ and let $g\in L^p_{0,1}(\C\mathbb{P}^1)$.
Then Cauchy's integral formula
$$I g_a (a):=\frac{1}{2\pi i}\int_{\Delta_1(0)} g_a(t) \frac{d\o{t}\wedge dt}{t-a}$$
is a continuous $\dq$-solution operator $L^p(\Delta_1(0)) \rightarrow C^\alpha(\Delta_1(0))$
on the unit disc $\Delta_1(0)$ for all $\alpha\in (0,1-\frac{2}{p}]$ (cf. \cite{He}, Hilfssatz 15).
Moreover, $Ig_a$ is continuous on $\C$ and $Ig_a(a)\rightarrow 0$ for $|a|\rightarrow \infty$.
So,
$$f_a\big(w:z\big):=(\Psi_a^{-1})^* I g_a \big(w:z\big) = I g_a \left(\frac{z}{w}\right)$$
is a continuous function on $\C\mathbb{P}^1$ which satisfies $\dq f_a = g$ on $\Psi_a(\Delta_1(0))$
and is holomorphic outside $\Psi_a(\o{\Delta_1(0)})$. We can now define a composed
integral operator:
\begin{eqnarray*}
S g \big(w:z\big) &:=& I g_a\left(\frac{z}{w}\right) + I g_b\left(\frac{w}{z}\right),
\end{eqnarray*}
which is a continuous operator $L^p_{0,1}(\C\mathbb{P}^1) \rightarrow C^\alpha(\C\mathbb{P}^1)$
for all $\alpha\in (0,1-\frac{2}{p}]$. Let 
$$M:=\{[w:z]\in\C\mathbb{P}^1: |w|/|z|=1\}=b\Psi_a(\Delta_1(0))=b\Psi_b(\Delta_1(0)).$$
Then $\dq Sg=g$ on $\C\mathbb{P}^1\setminus M$. But $Sg$ is continuous and
Theorem \ref{thm:extension2} guarantees $\dq Sg=g$ on the whole projective space.
We summarize:

\begin{thm}
Let $2<p<\infty$. The integral operator
\begin{eqnarray*}
Sg\big(w:z\big) = \frac{1}{2\pi i} \int_{\{|t|<1\}} \frac{\Psi_a^* g (t) \wedge dt}{t-z/w}
&+& \frac{1}{2\pi i} \int_{\{|t|<1\}} \frac{\Psi_b^* g (t) \wedge dt}{t-w/z}
\end{eqnarray*}
is a bounded operator $L^p_{0,1}(\C\mathbb{P}^1)\rightarrow C^\alpha(\C\mathbb{P}^1)$
for all $\alpha\in (0,1-\frac{2}{p}]$,
and it is a solution operator for the $\dq$-equation:  $\dq S g=g$.
\end{thm}

\newpage
\section{The Normal Bundle of Exceptional Sets}

Let $Y$ be a pure $d$-dimensional analytic set in $\C^n$ with an isolated singularity
at the origin and $\pi: M \rightarrow Y$ a desingularization as in the introduction.
Then, $Y$ is the Remmert reduction of $M$, and $M$ is a $1$-convex space. 
See \cite{CoMi} for a nice characterization of $1$-convex spaces.
$X=\pi^{-1}(\{0\})$ is the union of finitely many $(d-1)$-dimensional compact complex manifolds
$X_1, ..., X_k$ with only normal crossings. $X=\bigcup X_j$ is an exceptional set in the
sense of Grauert (\cite{Gr1}). 
The aim of this section is to provide an embedding
of a neighborhood of the zero section of the normal bundle of the $X_j$ into $M$, 
which is at least possible if $X$ is regular or if the dimension $d=2$.

Generally, if $X$ is a complex subspace of a complex space $Z$ with invertible ideal sheaf
$\mathcal{I}$ given locally by $\mathcal{I}|_{U_\alpha} = g_\alpha \mathcal{O}_Z|_{U_\alpha}$,
then the normal bundle of $X$ in $Z$ is the line bundle determined by the transition functions
$g_\alpha/g_\beta$ in $U_\alpha\cap U_\beta \cap X$.

Let $X$ be a compact complex manifold, $p: L\rightarrow X$ a holomorphic line bundle, and
$c(L)$ its Chern class in $H^2(X,\R)$. Then $L$ is called positive (negative) in the sense
of Kodaira, if $c(L)$ is represented locally by a real closed differential form 
$\omega= i \sum g_{\nu\o{\mu}} dz_\nu\wedge d\o{z_{\mu}}$ such that $(g_{\nu\o{\mu}})$
is positive (negative) definite. 
This notion has been generalized by Grauert (\cite{Gr1}, Definition 3.1)
to the following:

\begin{defn}
A holomorphic vector bundle $F$ over a complex space $Z$ is called weakly negative
if there exists a strongly pseudoconvex neighborhood $U(\mathfrak{0})$ of the zero section
$\mathfrak{0}$ in $F$. It is called weakly positive if its dual $F^*$ is weakly negative.
\end{defn}

A holomorphic line bundle over a complex manifold is negative exactly if it is weakly negative,
and that is the case exactly if its zero section is an exceptional set in the sense of Grauert.\\

But, there are weakly negative vector bundles
over compact complex manifolds that are not negative in the sense of Kodaira. 
Our basic intention is to use results of the following kind (\cite{Gr1}, Satz 4.6):

\begin{thm}\label{thm:embedding}
Let $M$ be a complex manifold and $X\subset M$ a compact complex submanifold
with weakly negative normal bundle $N$. Then there is a neighborhood $U(\mathfrak{0})$
of the zero section $\mathfrak{0}$ in $N$ and a biholomorphic mapping 
$\Psi: U(\mathfrak{0}) \rightarrow M$ such that $\Psi(\mathfrak{0})=X$.
\end{thm}

This is an application of Grauert's formal principle which has been generalized to much more general
settings. It is known that the normal bundle of the exceptional set $X=\bigcup X_j$ is weakly negative
(cf. \cite{HiRo}, Lemma 10), but how about the normal bundle of one irreducible component $X_j$?
We can give a simple answer if $M$ is a complex surface and the $X_j$ are regular complex curves:

\begin{lem}\label{lem:embedding}
Let $Y$ be an analytic subset of pure dimension $2$ of an open subset in $\C^n$
with one isolated singularity.
If $X\subset M$ is an irreducible component of the exceptional set of a desingularization $\pi: M\rightarrow Y$,
then the normal bundle $N$ of $X$ in $M$ is negative.
Hence, there exists a neighborhood $U(\mathfrak{0})$ of the zero section $\mathfrak{0}$ in $N$
and a biholomorphic mapping $\Psi: U(\mathfrak{0}) \rightarrow M$ such that $\Psi(\mathfrak{0})=X$.
\end{lem}

\begin{proof}
We assume that $Y$ has an isolated singularity at the origin $0\in\C^n$.
Let $A=\pi^{-1}(\{0\})$ be the exceptional set of the desingularization $\pi: M \rightarrow Y$.
If $A$ consist of only one irreducible component $X=A$, then we are done by the result
of Hironaka and Rossi (\cite{HiRo}, Lemma 10). We have to treat the case where $A$ consists
of more than one regular complex curves with only normal crossings.

So, let $X$ be an irreducible component of the exceptional set $A$ and $N\rightarrow X$
its normal bundle.
There exists an index $1\leq \nu \leq n$ such that $f:=z_\nu\circ \pi$ does not vanish
identically on $M$. $f$ is vanishing of some finite order $k$ along $X$. Hence, $f$ defines
a holomorphic section $s$ in the $k$-th power of the dual bundle $(N^*)^k$ (take $f/g_\alpha$), which is not
identically zero. Let
$$C:=\o{\{w\in M: f(w)=0\}\setminus X}.$$
Then $f$ is vanishing of higher order on $X\cap C$, which is not empty by assumption
(that there are more than one irreducible components of the exceptional set).
So, $s$ is a holomorphic section in $(N^*)^k$ that is not identically zero,
but vanishing in $X\cap C\neq \emptyset$. Now, we can use of the following
criterion given by Grauert in the proof of \cite{Gr1}, Satz 3.4:

\begin{lem}
Let $L$ be a holomorphic line bundle over an irreducible compact complex space of
dimension one. Assume that there is an integer $k\geq 1$ such that there exists
a holomorphic section $s$ in $L^k$ which is not vanishing identically but has at least
one zero. Then $L$ is positive.
\end{lem}

We conclude that $N^*$ is a positive line bundle. Hence, its dual bundle $N$ is negative.
So, it is also weakly negative and Theorem \ref{thm:embedding} gives the desired embedding
of a neighborhood of the zero section into $M$.
\end{proof}

\section{Weighted Cohomology of Holomorphic Line Bundles}

Let $X$ be a complex manifold and $E \rightarrow X$ a negative holomorphic line bundle.
We denote by $\OO_E$ the structure sheaf of $E$ and by 
$\mathcal{I}$ the ideal sheaf of the zero section in $E$.
Then $\mathcal{I}^k\OO_E$ is a subsheaf of the sheaf of germs of
meromorphic functions $\mathcal{M}_E$ for each $k\in \Z$.
Let us repeat the construction of a fine resolution for the sheaves $\mathcal{I}^k\OO_E$.
We define for $1\leq p\leq\infty$
the presheaves
$$\mathcal{L}_{0,q}^p(U):=\{f\in L^{p,loc}_{0,q}(U): \dq f\in L^{p,loc}_{0,q+1}(U)\},\ \ U\subset E \mbox{ open},$$
which are already sheaves with the natural restriction maps.
In case $p=2$, we simplify the notation: $\mathcal{L}_{0,q}:=\mathcal{L}^2_{0,q}$.
For $k\in \Z$, we consider $\mathcal{I}^k \mathcal{L}^p_{0,q}$
as subsheaves of the sheaf of germs of differential forms with measurable coefficients.
Now, we define a weighted $\dq$-operator on $\mathcal{I}^k \mathcal{L}_{0,q}^p$.
Let $f \in (\mathcal{I}^k \mathcal{L}_{0,q}^p)_z$. Then $f$ can be written locally as $f = h^k f_0$,
where $h\in (\OO_E)_z$ generates $\mathcal{I}_z$ and $f_0\in (\mathcal{L}^p_{0,q})_z$. Let
$$\dq_k f := h^k \dq f_0 = h^k \dq ( h^{-k} f).$$
For $k\geq 0$, this operator is just the usual $\dq$-operator.
We obtain the sequence
$$0 \rightarrow \mathcal{I}^k\OO_E \rightarrow
\mathcal{I}^k \mathcal{L}^p_{0,0} \xrightarrow{\ \dq_k\ }
\mathcal{I}^k \mathcal{L}^p_{0,1} \xrightarrow{\ \dq_k\ }
\cdots \xrightarrow{\ \dq_k\ }
\mathcal{I}^k \mathcal{L}^p_{0,d} \rightarrow 0 ,$$
which is exact by the Grothendieck-Dolbeault Lemma and well-known regularity results ($d=\dim E=\dim X+1$).
It is a fine resolution of $\mathcal{I}^k \OO_E$ since the $\mathcal{I}^k \mathcal{L}_{0,q}^p$
are closed under multiplication by smooth cut-off functions.
Our next purpose is to compute
$$H^q(E,\mathcal{I}^k\OO_E) \cong \frac{\mbox{ker }
(\dq_k: \mathcal{I}^k\mathcal{L}^p_{0,q}(E) \rightarrow \mathcal{I}^k\mathcal{L}^p_{0,q+1}(E))}
{\mbox{Im }
(\dq_k: \mathcal{I}^k\mathcal{L}^p_{0,q-1}(E) \rightarrow \mathcal{I}^k\mathcal{L}^p_{0,q}(E))}$$
in terms of some cohomology groups on $X$.\\

If $\mu$ is a positive integer, let $E^\mu$ be the $\mu$-fold tensor product of $E$,
and $E^{-\mu}$ the $\mu$-fold tensor product of the dual bundle $E^*$, $E^0$ the trivial line bundle.
We write $\OO(E^\mu)$ for the sheaf of germs of holomorphic sections in $E^\mu$.
Do not confuse $\OO(E^\mu)$ with the space of globally holomorphic functions on $E$, 
which is $\OO_E(E)=\Gamma(E,\OO_E)$ in our notation. We have the following expansion for cohomology classes:

\begin{thm}\label{thm:cohom}
Let $X$ be a compact complex manifold and $E\rightarrow X$ a negative holomorphic line bundle, and $k\in \Z$.
Let $D\subset \subset E$ be a strongly pseudoconvex smoothly bounded neighborhood of the zero section in $E$, or $D=E$.
Then there exist a natural isomorphism
\begin{eqnarray}\label{eq:cohom}
i_q: H^q(D,\mathcal{I}^k \OO_E) \rightarrow
\bigoplus_{\mu\geq k} H^q(X,\OO(E^{-\mu}))\ \ \mbox{ for all } q\geq 1.
\end{eqnarray}
\end{thm}


Before we prove Theorem \ref{thm:cohom}, it is convenient to introduce another point of view.
Since a sheaf $\mathcal{I}^k\OO_E$ is locally free of rank $1$ (invertible), there is a corresponding holomorphic line bundle $B^k \rightarrow E$
(defined by quotients of generators of $\mathcal{I}^k$ as transition functions) such that
\begin{eqnarray*}
H^q(U, \mathcal{I}^k\OO_E) \cong H^q(U, B^k)
\end{eqnarray*}
for all $q\geq 0$, and $U\subset E$ open. Note that $B^0\cong E\times \C$ is the trivial bundle, and 
$$B^k|_X \cong E^{-k}$$
for all $k\in \Z$, because of the transition functions. Now, we are in the position to apply classical results
about holomorphic functions with values in holomorphic vector bundles in the proof of Theorem \ref{thm:cohom}:

\begin{proof}
Let $d=\dim E=\dim X+1$. We only need to consider $1\leq q\leq d-1$,
because both sides in \eqref{eq:cohom} vanish if $q>d$. If $q=d$,
the right hand side in \eqref{eq:cohom} vanishes for dimensional reasons,
whereas the left hand side vanishes by a result of Siu (see the main statement in \cite{Siu}),
because $E$ does not contain $d$-dimensional compact connected components.\\

Moreover, we can reduce the case $D=E$ to the situation where $D\subset\subset E$ is a strongly pseudoconvex neighborhood of the zero section,
because blowing down the zero section of $E$ to a point gives an affine algebraic space (see \cite{Gr1}, section 3, Satz 5).
We can assume that this point is $0\in \C^n$. Then, the pull-back of $\|\cdot\|^2$ defines an exhaustion function $\rho$ of $E$
which is strictly plurisubharmonic outside the zero section, and has strongly pseudoconvex level sets.
For some $\epsilon>0$, let $D=\{z\in E: \rho(z)<\epsilon\}$.
Then $E$ is a $(d-1)$-convex extension of $D$ in the sense of Henkin and Leiterer (\cite{HeLe2}, Definition 12.1),
where $d=\dim E$. Hence the restriction map
\begin{eqnarray*}
H^q(E,B^k) \rightarrow H^q(D,B^k)
\end{eqnarray*}
is an isomorphism for all $1\leq q \leq d$ and $k\in\Z$.
Hence
\begin{eqnarray*}
H^q(E, \mathcal{I}^k\OO_E) \cong H^q(E, B^k) \cong H^q(D,B^k) \cong H^q(D,\mathcal{I}^k\OO),
\end{eqnarray*}
so that we only have to consider the case that $D\subset\subset E$ is strongly pseudoconvex.\\

We will compute the cohomology groups directly as Dolbeault cohomology groups.
Let $\mathcal{C}^\infty_{0,q}$ be the sheaf of germs of smooth $(0,q)$-forms.
For
$$0 \rightarrow \mathcal{I}^k\OO_E \xrightarrow{\ \iota\ }
\mathcal{I}^k \mathcal{C}^\infty_{0,0} \xrightarrow{\ \dq_k\ }
\mathcal{I}^k \mathcal{C}^\infty_{0,1} \xrightarrow{\ \dq_k\ }
\cdots \xrightarrow{\ \dq_k\ }
\mathcal{I}^k \mathcal{C}^\infty_{0,d} \rightarrow 0$$
is also a fine resolution of $\mathcal{I}^k\OO_E$,
it is enough to work in the $C^\infty$-category.\\

So, let the class $[\omega]\in H^q(D,\mathcal{I}^k \OO_E)$ be represented by
the $\dq_k$-closed $(0,q)$-form $\omega\in \mathcal{I}^k \mathcal{C}^\infty_{0,q}(D)$.
Because $D$ is strongly pseudoconvex, we can assume that $\omega$ extends $\dq_k$-closed to a neighborhood of $D$. So, $\omega\in \mathcal{I}^k\mathcal{C}^\infty_{0,q}(\o{D})$.

Let $\{U_\alpha\}_{\alpha\in\mathcal{A}}$ be an open covering of $X$ such that $E|_{U_\alpha}$ is trivial.
Then $\omega$ is represented locally by $\{\omega_\alpha\}_{\alpha\in\mathcal{A}}$.
On $E|_{U_\alpha}\cong U_\alpha \times \C$ we have local coordinates $z_1, ..., z_{d-1}, s$,
and there is a unique decomposition 
$$\omega_\alpha = \omega_\alpha^I + \omega_\alpha^{II}\wedge d\o{s},$$
such that $\omega_\alpha^I$ does not contain $d\o{s}$.
If $U_\alpha\cap U_\beta\neq \emptyset$, then $z_1', ..., z_{d-1}', t$ are local coordinates on $E|_{U_\beta}$
with
$$ s = g_{\alpha\beta}(z) t\ ,\ \ g_{\alpha\beta}\in \Gamma(U_\alpha\cap U_\beta,\OO^*).$$
We have $\omega_\alpha^I=\omega_\beta^I$ and  $\omega_\alpha^{II}\wedge d\o{s}=\omega_\beta^{II}\wedge d\o{t}$.
This leads to a global decomposition $\omega=\omega^I + \omega^{II}$
where $\omega^I$ is given by the $\{\omega_\alpha^I\}$ and $\omega^{II}$ is given
by the $\{\omega_\alpha^{II}\wedge d\o{s}\}$. First of all,
we will show that we may assume $\omega^{II}=0$.
For this, we integrate along the fibers of $D$. So, if $(z,a)\in E|_{U_\alpha}\cong U_\alpha\times \C$,
let
$$P(f d\o{s}) (z,a) := \frac{1}{2\pi i} \int_{D\cap E_z} f(z,s) \frac{d\o{s}\wedge ds}{s-a}.$$
Then, for
$$u = f(z,s)\ d\o{z_{l_1}}\wedge...\wedge d\o{z_{l_{q-1}}}\wedge d\o{s}$$
we set
$$P u := P \big(f(z,s)d\o{s}\big)\ d\o{z_{l_1}}\wedge...\wedge d\o{z_{l_{q-1}}},$$
where $P$ operates in $s$ on $f(z,s)d\o{s}$. 
Because the boundary of $D$ is smooth,
$P \big(f(z,s)d\o{s}\big)$ is also smooth.
The operator ${\bf P}$ extends linearly to arbitrary forms.
Since $\omega_\alpha^{II}\wedge d\o{s}=\omega_\beta^{II}\wedge d\o{t}$ on $U_\alpha\cap U_\beta\neq \emptyset$,
the locally defined  forms 
$$\{\eta_\alpha=s^k P(s^{-k} \omega_\alpha^{II}\wedge d\o{s})\}$$
define a global $(0,q-1)$-form $\eta={\bf P}\omega \in \mathcal{I}^k \mathcal{C}^\infty_{0,q}(E)$.
Note that the $\eta_\alpha$ do not contain $d\o{s}$ 
and that
$$s^k \frac{\partial}{\partial\o{s}} \left( s^{-k} \eta_\alpha\right)= \omega_\alpha^{II}.$$

Now, $\wt{\omega}:=\omega+(-1)^q \dq_k \eta$ lies in the same $\dq_k$-cohomology class as $\omega$,
but in local coordinates, it does not contain $d\o{s}$.
So, from now on, we will assume that $\omega$ is represented by $\{\omega_\alpha+\omega_\alpha^{II}\wedge d\o{s}\}$
where all the $\omega_\alpha^{II}=0$.
Due to linearity, it is enough to consider
$$\omega_\alpha(z,s) = f(z,s)\ d\o{z_1}\wedge ... \wedge d\o{z_q}.$$
Since $\dq_k \omega=0$ and we have eliminated the $d\o{s}$-parts,
we have
$$s^k\frac{\partial }{\partial \o{s}} \left(s^{-k} f\right)=0,$$
which means that the coefficients $s^{-k} f$ are holomorphic in $s$.
This implies the local expansion
\begin{eqnarray}\label{eq:expansion}
\omega_\alpha(z,s) = \sum_{\mu\geq k} a_{\alpha,\mu}(z)\ s^\mu,
\end{eqnarray}
where the $a_{\alpha,\mu}$ are smooth $\dq$-closed $(0,q)$-forms on $U_\alpha\subset X$.

Now, since $s^\mu = g_{\alpha\beta}(z)^\mu\ t^\mu$ on $U_\alpha\cap U_\beta\neq\emptyset$,
we deduce the transformation rule
$$a_{\alpha,\mu} (z) = g_{\alpha\beta}(z)^{-\mu} a_{\beta,\mu}(z).$$
This means that the coefficients $\{a_{\alpha,\mu}\}$ are nothing else but
globally defined $(0,q)$-forms $a_\mu$ on $X$ with coefficients  in $E^{-\mu}$ (for $\mu\geq k$),
and that $\omega$ is $\dq_k$-closed exactly if all the $a_\mu$, $\mu\geq k$, are $\dq$-closed.
It is easy to see that a different representing form for $[\omega]$ defines the same classes
$[a_\mu] \in H^q(X,\OO(E^{-\mu}))$: If $\eta\in \mathcal{I}^k \mathcal{C}^\infty_{0,q-1}(D)$
with $\dq_k \eta=\omega$,
then the coefficients of $s^{-k} \eta_\alpha$ are holomorphic in $s$ which implies the local expansion
$$\eta_\alpha(z)=\sum_{\mu\geq k} c_{\alpha,\mu}(z)\ s^\mu$$
with $\dq c_{\alpha,\mu} = a_{\alpha,\mu}$.
We have thus defined a homomorphism
$$i_q: H^q(D,\mathcal{I}^k \OO_E) \rightarrow
\prod_{\mu\geq k} H^q(X,\OO(E^{-\mu})).$$
The fact that the $a_{\alpha,\mu}$ transform inversely to the linear coordinates of the line bundle
implies that this construction is independent of the trivialization.
Now, we will use the fact that
$E$ is a negative holomorphic line bundle over the compact complex manifold $X$.
So, for any coherent analytic
sheaf $\mathcal{S}$ over $X$, there exists an integer $\mu_0\geq 1$,
such that
\begin{eqnarray}\label{eq:grauert1}
H^q(X,\mathcal{S} \otimes \OO(E^{-\mu}))=0
\end{eqnarray}
for all $q\geq 1$, $\mu\geq \mu_0$ (cf. \cite{Gr1}). Taking simply $\mathcal{S}=\OO_X$,
it follows that
$$\prod_{\mu\geq k} H^q(X,\OO(E^{-\mu})) = \bigoplus_{\mu\geq k} H^q(X,\OO(E^{-\mu}))$$
is finite-dimensional, and it is clear that $i_q$ is surjective. For, if 
$$[\theta] \in \bigoplus_{\mu\geq k} H^q(X,\OO(E^{-\mu})),$$
then there exists $j_0\geq k$ such that $[\theta]$ is represented locally by
$$\theta=(\{\theta_{\alpha,k}\}_{\alpha\in\mathcal{A}},\{\theta_{\alpha,k+1}\}_{\alpha\in\mathcal{A}},...),$$
where $\theta_{\alpha,j}\equiv 0$ for all $j> j_0\geq k$, $\alpha\in\mathcal{A}$.
But then a form $\omega\in \mathcal{I}^k\mathcal{C}^\infty_{0,q}(E)$ is well-defined by the the finite sum
$$\omega_\alpha(z,s) = \sum_{\mu= k}^{j_0} a_{\alpha,\mu}(z)\ s^\mu$$
such that $\dq_k \omega=0$ and $i_q([\omega])=[\theta]$ showing that $i_q$ is onto.
It remains to prove that $i_q$ is also injective which is a bit more complicated.
Let $Z$ be the Remmert reduction of $E$, $\pi: E\rightarrow Z$ the projection.
Then $\pi: E\rightarrow Z$ is a desingularization of $Z$ with exceptional set $X$,
and $D$ is the pre-image of a Stein domain in $Z$.
So, we can use the results of Forn{\ae}ss, {\O}vrelid and Vassiliadou in \cite{FOV1}.

Namely, by \cite{FOV1}, Proposition 1.3, there exists an index $\nu\geq k$
such that the natural map
\begin{eqnarray}\label{eq:FOV1}
j_*: H^q(D,\mathcal{I}^\nu \OO_E) \rightarrow H^q(D,\mathcal{I}^k \OO_E)
\end{eqnarray}
induced by the inclusion $j: \mathcal{I}^\nu \OO_E \rightarrow \mathcal{I}^k \OO_E$ is the zero map.
So, for a class $[\omega]\in H^q(D,\mathcal{I}^k\OO_E)$ as in the beginning, let us return to the expansion
\eqref{eq:expansion}.
We simply cut off the first $\nu-k$ summands.
Namely, let $\wt{\omega} \in \mathcal{I}^k C^\infty_{0,q}(E)$ be given locally by
\begin{eqnarray}\label{eq:expansion2}
\wt{\omega}_\alpha(z,s) = \sum_{\mu =  k}^{\nu-1} a_{\alpha,\mu}(z)\ s^\mu.
\end{eqnarray}
Then $\wt{\omega}$ is $\dq_k$-closed, and $\omega-\wt{\omega}$ defines a class in $H^q(D,\mathcal{I}^\nu \OO_E)$.
By \eqref{eq:FOV1}, this implies that $[\omega-\wt{\omega}]=0$ in $H^q(D,\mathcal{I}^k\OO_E)$, and that
$[\omega]=[\wt{\omega}] \in H^q(D,\mathcal{I}^k\OO_E)$.
Now, because \eqref{eq:expansion2} is finite, it follows that $\wt{\omega}$ is $\dq_k$-exact exactly if the coefficients
$a_\mu$ are exact for all $k\leq\mu\leq \nu-1$, which means that $i_q$ is injective.
\end{proof}

We will conclude this section by a few interesting
consequences in the given situation.
Theorem \ref{thm:cohom} leads to the observation that there exists $\mu_0\geq 1$ such that
$$H^q(E^{\mu_0},\mathcal{I}^k\OO_{E^{\mu_0}}) 
\cong \bigoplus_{\mu\geq k} H^q(X,\OO(E^{-\mu_0\mu})) = 0\ \ \mbox{for all } q\geq 1,$$
provided $k$ is positive. In case $k\leq 0$, the situation is not that easy.
But we are in the position to use Serre duality.
Let $\Omega^{d-1}$ be the canonical sheaf on $X$, where $d=\dim X+1$.
Then:
\begin{eqnarray}\label{eq:grauert2}
H^q(X,\OO(E^\mu)) \cong H^{(d-1)-q}(X,\Omega^{d-1}\otimes\OO(E^{-\mu}))
\end{eqnarray}
for all $0\leq q\leq d-1$.
Combining this duality with Grauert's Vanishing Theorem \eqref{eq:grauert1}
for $\mathcal{S}=\Omega^{d-1}$ on the right hand side of \eqref{eq:grauert2},
it follows that there exists an integer
$\mu_1\geq 1$ such that
$$H^q(X,\OO(E^\mu))=0 \ \ \mbox{ for all } 0\leq q\leq d-1, \mu\geq \mu_1.$$
Taking $\nu_0=\max\{\mu_0,\mu_1\}$, Theorem \ref{thm:cohom} implies:

\begin{thm}\label{thm:cohom2}
Let $X$ be a compact complex manifold, and let $E\rightarrow X$ be
a negative holomorphic line bundle. Then there exists an integer $\nu_0\geq 1$,
such that the following is true for all $\nu\geq\nu_0$: Let $D=E^\nu$ or $D\subset\subset E^\nu$
a smoothly bounded strongly pseudoconvex neighborhood of the zero section in $E^\nu$.
Then  $H^q(D,\OO_{E^\nu})\cong H^q(X,\OO_X)$ for all $q\geq 1$.
If $k$ is a positive integer, then
$$H^q(D,\mathcal{I}^k \OO_{E^\nu})=0 \ \ \ \mbox{ for all }\  q\geq 1.$$
If $k<0$, then
$$H^q(D,\mathcal{I}^k \OO_{E^\nu})\cong H^q(X,\OO_X) \ \ \ \mbox{ for all }\ 1\leq q\leq \dim X-1.$$
\end{thm}

\newpage
\section{A Sufficient Condition for $L^2$-Solvability \\ (Theorem \ref{thm:l2})}

This section is dedicated to the proof of Theorem \ref{thm:l2}.
So, let $Y$ be a pure dimensional analytic set in $\C^n$
with an isolated singularity at the origin
such that the exceptional set $X$ of the desingularization
$\pi: M\rightarrow Y$ is regular. Furthermore, let $\Omega\subset\subset Y$ be a Stein domain
with $0\in \Omega$, $\Omega^*=\Omega\setminus \{0\}$ and $\Omega'=\pi^{-1}(\Omega)$.
Recall that $Y^*=Y\setminus\{0\}$ carries the induced hermitian metric 
of the embedding $\iota: Y^*\hookrightarrow\C^n$, and give $M$ an arbitrary fixed hermitian metric.\\

Let $z\in X$. Then, by Lemma \ref{lem:detjac},  
there exists a neighborhood $U_z$ in $M$ and a holomorphic function
$J_z\in\OO(U_z)$ such that $|J_z(w)|^2=\det\Jac \pi(w)$ for all $w\in U_z\setminus X$.
It follows that all such representing holomorphic functions vanish of a fixed order $k_0\geq 1$
along the irreducible analytic set $X$.
This is the index from the statement of Theorem \ref{thm:l2}.
So, we assume that
\begin{eqnarray}\label{eq:assumption}
H^1(X,\OO(N^{-\mu}))=0
\end{eqnarray}
for all $\mu\geq -k_0$.
Let $\mathcal{I}$ be the ideal sheaf of $X$ in $M$. Recall the fine resolution
\begin{eqnarray}\label{eq:resolution}
0 \rightarrow \mathcal{I}^k\OO_M \rightarrow
\mathcal{I}^k \mathcal{L}_{0,0} \xrightarrow{\ \dq_k\ }
\mathcal{I}^k \mathcal{L}_{0,1} \xrightarrow{\ \dq_k\ }
\cdots \xrightarrow{\ \dq_k\ }
\mathcal{I}^k \mathcal{L}_{0,d} \rightarrow 0.
\end{eqnarray}
We are now looking for a neighborhood $V\subset \Omega'$ of $X$
such that 
$$H^1(V,\mathcal{I}^{-k_0}\OO_M) \cong \bigoplus_{\mu\geq -k_0} H^1(X,\OO(N^{-\mu})).$$
So, let $N\rightarrow X$ be the normal bundle of $X$ in $M$.
Then, by Theorem \ref{thm:embedding},
we know that there is a neighborhood $U(\mathfrak{0})$ of the zero section in $N$
and a biholomorphism $\Psi: U(\mathfrak{0})\rightarrow M$ such that $\Psi(\mathfrak{0})=X$.

But $\rho:=\|\cdot\|^2\circ \pi$ is an exhaustion function of $M$ which is strictly plurisubharmonic
outside $X$ and vanishes exactly on $X$.
So, there exists $\epsilon>0$ such that
$$V:=\{z\in M: \rho(z)<\epsilon\} \subset\subset \Psi(U(\mathfrak{0}))\cap \Omega'$$
is strongly pseudoconvex. Hence,
$$D:=\Psi^{-1}(V)\subset\subset U(\mathfrak{0})$$
is a strongly pseudoconvex neighborhood  of the
zero section in $N$. Since $\Psi$ is a biholomorphism,
Theorem \ref{thm:cohom} leads to:
\begin{eqnarray}\label{eq:vanishing}
H^1(V,\mathcal{I}^{-k_0}\OO_M)\cong H^1(D,\mathcal{I}^{-k_0}\OO_{N}) \cong \bigoplus_{\mu\geq -k_0} H^1(X,\OO(N^{-\mu})) =0
\end{eqnarray}
by assumption \eqref{eq:assumption}.
Now, let $\omega\in L^2_{0,1}(\Omega^*)$ be $\dq$-closed. Then, there is a representation
$$\omega = \sum_{j=1}^n \iota^* \big( f_j d\o{z_j}\big)$$
where $f_j\in L^2(\Omega^*)$ for $j=1, ..., n$.\\

Since $\pi: M\rightarrow Y\subset\C^n$
is a holomorphic map, the $\pi^*\iota^* d\o{z_j}$ are bounded forms on $M$.
By Lemma \ref{lem:detjac} and the considerations above, 
$\pi^* f_j \in \mathcal{I}^{-k_0}\mathcal{L}_{0,0}(\Omega')$, and it follows
that $\pi^*\omega\in \mathcal{I}^{-k_0} \mathcal{L}_{0,1}(\Omega')$. $\pi^* \omega$
is $\dq$-closed on $M\setminus X$ and locally square integrable if it 
is multiplied with the $k_0$-fold product of a local generator of $\mathcal{I}$.
So, the $\dq$-Extension Theorem \ref{thm:extension} implies that $\dq_{-k_0} \pi^*\omega=0$.
Now, we can apply the vanishing result \eqref{eq:vanishing}. Taking into account
that \eqref{eq:resolution} is a fine resolution, it follows that there exists
$\eta'\in \mathcal{I}^{-k_0}\mathcal{L}_{0,0}(V)$ such that
$\dq_{-k_0} \eta' = \pi^* \omega$ on $V$. Let $V'\subset\subset V$ be a smaller neighborhood of $X$.
Then, by Lemma \ref{lem:detjac},
this means that we have found
$$\eta_0:=(\pi|_{M\setminus X}^{-1})^* \eta' \in L^2(\pi(V')^*)\ , \ \dq\eta_0=\omega \ \mbox{ on } \ \pi(V')^*.$$

This solution can be extended to the whole set $\Omega^*$ by the use
of H\"ormander's $L^2$-theory (see \cite{Hoe}). $\Omega'\subset M$ is a $1$-convex space.
Hence, it carries a strictly plurisubharmonic exhaustion function $\phi: \Omega'\rightarrow[-\infty,\infty)$,
which can be chosen $-\infty$ exactly on $X$ and real analytic outside $X$ (cf. \cite{CoMi}).
Choose $c_0\in\R$ such that 
$$\Omega_{c_0}:=\{z\in\Omega': \phi(z)<c_0\} \subset\subset V'.$$
Let $c_1<c_0$ and $\Phi\in C^\infty(\Omega')$ be a modification of $\phi$
such that $\Phi\equiv\phi$ on $\Omega'\setminus \o{\Omega_{c_1}}$
and $\Phi$ is an exhaustion function of $\Omega'$ which is strictly plurisubharmonic
on $\Omega'\setminus \o{\Omega_{c_1}}$. 
Let $\chi\in C^\infty(M)$ be a smooth cut-off function with compact support in $V'$
which is identically 1 in a neighborhood of $\o{\Omega_{c_0}}$. Let
$$\wt{\omega} := \pi^*\omega - \dq(\chi \eta').$$
Then $\wt{\omega}$ is $\dq$-closed, 
square integrable on $\Omega'$ and identically zero in a neighborhood of $\o{\Omega_{c_0}}$.
So, it follows by \cite{Hoe}, Theorem 3.4.6, 
as in the proof of Proposition 5.1 in \cite{FOV1},
that  there is a solution $\wt{\eta}\in L^2(\Omega')$
such that $\dq \wt{\eta}=\wt{\omega}$ on $\Omega'$. But
$(\pi|_{M\setminus X}^{-1})^* \wt{\eta} \in L^2(\Omega^*)$ by Lemma \ref{lem:detjac}
and
$$\eta_1 := (\pi|_{M\setminus X}^{-1})^* \big(\chi \eta' + \wt{\eta}\big) \in L^2(\Omega^*)$$
is the desired solution $\dq \eta_1=\omega$ on $\Omega^*$.

Now, standard arguments (cf. \cite{FOV1}, Lemma 4.2) imply that
there is a constant $C_\Omega>0$ such that for each $\omega\in L^2_{0,1}(\Omega^*)$
that is $\dq$-closed on $\Omega^*$, there is $\eta\in L^2(\Omega^*)$ 
with $\dq\eta=\omega$ on $\Omega^*$ and
$$\|\eta\|_{L^2(\Omega^*)} \leq C_\Omega \|\omega\|_{L^2_{0,1}(\Omega^*)}.$$

This completes the proof of Theorem \ref{thm:l2}.

\section{The $\dq$-Equation for Bounded $(0,1)$-Forms}

Here again let $Y$ be a pure dimensional analytic set in $\C^n$
with an isolated singularity at the origin
such that the exceptional set $X$ of the desingularization
$\pi: M\rightarrow Y$ is regular. Furthermore, let $\Omega\subset\subset Y$ be strongly pseudoconvex
with $0\in \Omega$, $\Omega^*=\Omega\setminus \{0\}$ and $\Omega'=\pi^{-1}(\Omega)$.

We will now prove Theorem \ref{thm:c0}. Our strategy is quite similar to the proof
of Theorem \ref{thm:l2}: As before, Theorem \ref{thm:embedding} implies
that there is a neighborhood $V\subset\Omega'$ of $X$ in $M$ such that
$$H^1(V,\mathcal{I}^k \OO_M) \cong H^1(N, \mathcal{I}^k \OO_N)$$
for all $k\in\Z$, where $N$ is the normal bundle of $X$ in $M$.  
But this time,
we will be in a position to use Theorem \ref{thm:cohom} in case $k=1$.
Hence, we will make use of
\begin{eqnarray}\label{eq:h11}
H^1(V,\mathcal{I}^1 \OO_M) \cong \bigoplus_{\mu\geq 1} H^1(X, \mathcal{I}^1 \OO(N^{-\mu}),
\end{eqnarray}
where the right hand side vanishes by assumption.\\

Let $\omega\in L^\infty_{0,1}(\Omega^*)$ be $\dq$-closed. 
In order to use \eqref{eq:h11},
we have to show that we can assume $\pi^*\omega\in \mathcal{I}^1\mathcal{L}_{0,1}(V)$.
This is done by a certain regularization procedure. 
There is a representation
$$\omega = \sum_{j=1}^n \iota^* \big( f_j d\o{z_j}\big)$$
where $f_j\in L^\infty(\Omega^*)$ for $j=1, ..., n$, and
$$\max_{j=1, ..., n} \|f_j\|_{L^\infty(\Omega^*)} \leq \|\omega\|_{L^\infty_{0,1}(\Omega^*)}.$$
Let $\Psi: B_r(0)=\{w\in\C^d:\|w\|<r\} \rightarrow \Omega'$ be a nice holomorphic ball
such that $\wt{X}:=\Psi^{-1}(X)=\{w\in B_r(0): w_1=0\}$, and $\Pi=\pi\circ\Psi$. We will always assume that
such a map $\Psi$ extends biholomorphically to a neighborhood of $B_r(0)$.
We will take a closer look at
$$\Pi^*\omega = \sum_{k=1}^d g_k\ d\o{w_k}.$$
Consider $\Pi=(\Pi_1, ..., \Pi_n): B_r(0) \subset \C^d \rightarrow Y \subset \C^n$.
Then the entries of the complex Jacobian of $\Pi$ as a mapping into $\C^n$ are holomorphic functions
$$\Pi_{jk}:=\frac{\partial \Pi_j}{\partial w_k}: B_r(0) \rightarrow \C,$$
and we compute
$$g_k(w) =\sum_{j=1}^n \Pi_{jk}(w) \cdot \Pi^* f_j (w) = \sum_{j=1}^n \Pi_{jk}(w)\cdot f_j(\Pi(w)).$$
Since $\Pi$ contracts $\wt{X}$ to the origin in $\C^n$,
it follows that the holomorphic functions $\Pi_{jk}$ vanish (at least) of order $1$
on $\wt{X}$ if $k\geq 2$.
This implies that for $k\geq 2$
$$|g_k(w)| \leq n\cdot  \sum_{j=1}^n |\Pi_{jk}(w)| \cdot \|\omega\|_{L^\infty_{0,1}(\Omega^*)}.$$
is vanishing of order $1$ along $\wt{X}$. 

Since $\Pi^*\omega \in L^\infty_{0,1}(B_r(0))$
is $\dq$-closed by the $\dq$-Extension Theorem \ref{thm:extension},
there exists a solution $\eta\in C^{1/2}(\o{B_r(0)})$, $\dq \eta=\Pi^*\omega$.
Let $w'=(w_2, ..., w_d)$. Then
$$\dq_{w'} \eta (w) = \sum_{k=2}^d g_k(w) d\o{w_k}\ \ \mbox{ on }\ B_r(0)\cap\{w_1=c\}$$
for almost all $c$. But here, the coefficients $g_k$, $k\geq 2$, vanish if $w_1\rightarrow 0$.
So, $\eta$ is a holomorphic function on $\wt{X}$ and 
$$\wt{\eta} (w):=\eta(w) - \eta(0,w_2, ..., w_d)\ \in C^{1/2}(\o{B_r(0)})$$
is also a solution $\dq \wt{\eta}=\Pi^*\omega$ on $B_r(0)$,
and vanishing of order $1/2$ on $\wt{X}$.\\

Let the neighborhood $V$ of $X$ be covered by $L$ holomorphic balls
$\Psi_l: B_r(0)\rightarrow M$,
and $\{\chi_l\}_{l=1}^L$ be a smooth partition of unity subordinate to that cover. Let
$$u_0:= \sum_{l=1}^L \chi_l\cdot (\Psi_l^{-1})^* \wt{\eta_l},$$
where the $\wt{\eta_l}$ are solutions of $\dq \wt{\eta_l} = \Psi^*_l \pi^* \omega$. 
By the considerations above,
we can assume that $u_0\in C^{1/2}(\o{V})$ is vanishing of order $1/2$ on $X$.
Shrink $V$ if necessary.
The continuous function $u_0$ is the first part of the solution of $\dq u=\pi^* \omega$ 
on $V$ we are looking for.
Let
$$\omega_0 := \sum_{l=1}^L \dq \chi_l \wedge (\Psi_l^{-1})^* \wt{\eta_l}.$$
Then $\dq u_0 = \pi^*\omega + \omega_0$, and $\omega_0\in C^{1/2}_{(0,1)}(\o{V})$ is $\dq$-closed
and vanishing of order $1/2$ on $X$.
We have thus reduced the problem to solving the $\dq$-equation for $\omega_0$.
Now, we will repeat this regularization procedure.

Let $\Psi: B_r(0) \rightarrow V$ be again a nice holomorphic ball such that
$$\wt{X}=\Psi^{-1}(X)=\{w\in B_r(0): w_1=0\}.$$
$\Psi^*\omega_0$ is bounded, $\dq$-closed and vanishing of order $1/2$ on $\wt{X}$.
This implies that $w_1^{-1}\cdot \Psi^*\omega_0$ is still square integrable,
and by the $\dq$-Extension Theorem \ref{thm:extension}, it is also $\dq$-closed.
So, there is a function $\eta_0\in L^2(B_r(0))$ such that $\dq\eta_0 = w_1^{-1}\Psi^*\omega_0$.
Let $\wt{\eta_0}:= w_1\cdot \eta_0$. 
Then $(\Psi^{-1})^*\wt{\eta_0} \in \mathcal{I}^1 \mathcal{L}_{0,0}(\Psi(B_{r'}(0)))$
for $0<r'<r$, and $\dq (\Psi^{-1})^*\wt{\eta_0} = \omega_0$ where it is defined.
So, repeat the construction of $u_0$ and $\omega_0$ with such local solutions:
$$u_1:= \sum_{l=1}^L \chi_l\cdot (\Psi_l^{-1})^* \wt{\eta_{0,l}}\ \ \mbox{ and } \ \ 
\omega_1 := \sum_{l=1}^L \dq \chi_l \wedge (\Psi_l^{-1})^* \wt{\eta_{0,l}}.$$
Then $\dq(u_0-u_1)=\pi^*\omega-\omega_1$, and
$\omega_1\in \mathcal{I}^1\mathcal{L}_{0,1}(V)$ is $\dq$-closed and $\dq_1$-closed.
Again, one may shrink $V$ if necessary. So, after two steps of the regularization procedure,
we have reduced the problem to solving the $\dq$-equation for $\omega_1\in \mathcal{I}^1\mathcal{L}_{0,1}(V)$.\\


But, as we have already realized (see \eqref{eq:h11}), $H^1(V,\mathcal{I}^1\OO_M)=0$.
Hence, there is $u_2\in \mathcal{I}^1\mathcal{L}_{0,0}(V)$ such that $\dq u_2=\dq_1 u_2=\omega_1$.
Then, $u:=u_0-u_1+u_2$ satisfies the equation $\dq u=\pi^*\omega$ on $V$.
But, as $\pi^*\omega$ is a bounded $(0,1)$-form, local regularity results for the $\dq$-equation
imply that $u\in C^0(V)$. Therefore, 
$$u_2-u_1\in \mathcal{I}^1\mathcal{L}_{0,0}(V)\cap C^0(V),$$
and this implies that $u_2-u_1$ is vanishing on $X$ (as is $u_0$). So, we have found
$$v':=(\pi^{-1})^* u \in C^0(\pi(V))\ \ , \ \dq v'=\omega\ \mbox{ on } \pi(V)^*.$$

As in the proof of Theorem 1.5.19 in \cite{HeLe1}, there exists a strictly plurisubharmonic function
$\rho$ in a neighborhood of $\o{\Omega}$ such that $\Omega=\{z\in Y: \rho(z)<1\}$ and 
$$\Omega_\epsilon=\{z\in Y: \rho(z)<\epsilon\}\subset \subset \pi(V)$$
for some $\epsilon>0$, $\epsilon$ a regular value of $\rho$.
By use of Grauert's bump method, it follows as in \cite{HeLe1}, Chapter 2.12,
that there is $v\in C^0(\overline{\Omega})$
with $\dq v=\omega$ on $\Omega^*$.

Again, \cite{FOV1}, Lemma 4.2, implies that
there is a constant $C_\Omega>0$ such that for each $\omega\in L^\infty_{0,1}(\Omega^*)$
that is $\dq$-closed on $\Omega^*$, there is $v \in C^0(\overline{\Omega})$ 
with $\dq\eta=\omega$ on $\Omega^*$ and
$$\|v\|_{C^0(\o{\Omega})} \leq C_\Omega \|\omega\|_{L^\infty_{0,1}(\Omega^*)}.$$


\newpage
\section{Homogeneous Varieties with Isolated Singularities}

In this section, we give the proof of Theorem \ref{thm:hoelder}, 
which is an easy consequence of Theorem \ref{thm:c0} and 
a $\dq$-integration formula for weighted homogeneous varieties (see \cite{RuZe}).
Before, we make a few remarks about the blow up of an isolated singularity
of a homogeneous variety for we do not know a suitable reference.

So, let $Y$ be an irreducible cone in $\C^n$ of pure dimension $d$ with an isolated singularity at the origin,
and $\Omega\subset\subset Y$ strongly pseudoconvex with $0\in\Omega$, and $\Omega^*=\Omega\setminus\{0\}$.
Then, as $Y$ is a homogeneous variety, it defines also a regular complex manifold $X'$ in $\C\mathbb{P}^{n-1}$.
Let $Y$ (respectively $X'$) be given by $L$ homogeneous holomorphic polynomials $P_1, ..., P_L$,
$$\Jac_1 P(z) := \left(\frac{\partial P_j}{\partial z_k}(z)\right)_{\substack{j=1, ..., L\\k=2, ..., n}}.$$
Then, by the Jacoby Criterion (cf. \cite{GrRe}), 
$\mbox{rank} \Jac_1 P(z)=n-d$, where $z_1\neq 0$
because $X'$ is regular in $\C\mathbb{P}^{n-1}$. Let's consider the blow up of the origin in $\C^n$:
$$A:=\{(z,w)\in \C^n\times\C\mathbb{P}^{n-1}: w_j z_k = w_k z_j \mbox{ for all } j\neq k\}$$
with the projection $\Pi: A \rightarrow \C^n, (z,w)\mapsto z$.
Let $M:=\o{\Pi^{-1}(Y\setminus\{0\})}$. 
Then $M$ is given in $A$ as the zero set of the polynomials $P_1, ..., P_L$ as
polynomials in $w$, and $\Pi^{-1}(\{0\})\cong X'$. To see this, note that in a chart 
$U_1:=\{(z,w):w_1=1\}$ we have
$P_j(z)=z_1^{\deg P_j} P(w)$,
and the component $\{z_1=0\}$ is not contained in $M$.\\

We claim that $M$ is a regular submanifold in $\C^n\times\C\mathbb{P}^{n-1}$.
That can be seen as follows:
It is enough to work in a chart $U_1:=\{(z,w):w_1=1\}$.
Then $M$ is defined in $U_1$ by the $n-1+L$ equations
$P_1(w)=0$, ..., $P_N(w)=0$, and $Q_j(z,w):=z_j-w_j z_1=0$ for $j=2, ..., n$.
Let $\wt{w}=(w_2, ..., w_n)$. Then 
$$\Jac_z Q(z,w) := \left(\frac{\partial Q_j}{\partial z_k}(z,w)\right)_{\substack{j=2, ..., n\\k=1, ..., n}}
=\big( \wt{w}^t\ E_{n-1}\big),$$
where $E_{n-1}$ is the unit matrix.
So, in order to check the Jacobi Criterion, we have to consider the following matrix:
$$\left(\begin{array}{cc}
\partial Q_j(w,z)/\partial z_k & \partial Q_j(w,z)/\partial w_l\\
\partial P_m(w)/\partial z_k & \partial P_m(w)/\partial w_l
\end{array}\right)_{\substack{j=2, .., n; m=1, .., L\\ k=1,.., n; l=2, .., n}} = 
\left(\begin{array}{cc}(\wt{w}^t\ E_{n-1}) & *\\
0 & \Jac_1 P(w)\end{array}\right),$$
which is of rank $(n-1)+(n-d)$. So, $M$ is in fact regular ($\dim M=d$),
and we see that $\pi:=\Pi|_M: M\rightarrow Y$ is a desingularization
with regular exceptional set $X:=\pi^{-1}(\{0\})\cong X'$.
Now, we can apply Theorem \ref{thm:c0}.\\

Assume that $\omega\in L^\infty_{0,1}(\Omega^*)$ is $\dq$-closed on $\Omega^*$, 
and we have found $v \in C^0(\overline{\Omega})$ such that $\dq v=\omega$ on $\Omega^*$ and
$$\|v\|_{C^0(\o{\Omega})} \leq C_\Omega \|\omega\|_{L^\infty_{0,1}(\Omega^*)},$$
where $C_\Omega>0$ is the constant from Theorem \ref{thm:c0}.
Now, choose a smooth cut-off function $\chi$ with compact support in $\Omega$ which 
is identically $1$ in a neighborhood of the origin.
Then,
$$\omega':= \dq (\chi v) \in L^\infty_{0,1}(\Omega^*)$$
is $\dq$-closed on $\Omega^*$ and has compact support in $\Omega$.
Then, by the use of Theorem 3 in \cite{RuZe}, it follows that 
there exists $g\in C^0(Y)$ such that the following is true:
$$\dq g = \omega'\ \ \mbox{ on } Y^*,$$
and for each $0<\theta<1$,
there exists a constant $C(\Omega)_\theta>0$
which does not depend on $\omega'$ and $g$, such that
$$\|g\|_{C^\theta(\o{\Omega})} \leq C(\Omega)_\theta\ \|\omega'\|_{L^\infty_{0,1}(\Omega^*)} 
\lesssim \|\omega\|_{L^\infty_{0,1}(\Omega^*)}.$$
So, let $\eta:= g + (1-\chi) v$.

Then $\dq \eta = \omega$ on $\Omega^*$. The function $(1-\chi) v$ is vanishing in a neighborhood
of the origin and is satisfying the $\dq$-equation with a bounded $(0,1)$-form
on the right hand side on the domain $\Omega$ which has a strongly pseudoconvex boundary.
Hence, $g$, $(1-\chi)v$ and consequently $\eta$, too,  belong to $C^{1/2}(\o{\Omega})$.
Once again, an application of  \cite{FOV1}, Lemma 4.2, finishes the proof of Theorem \ref{thm:hoelder}.\\

We will now investigate necessary conditions for solvability of the $\dq$-equation
in the $L^\infty$-sense on such varieties.
So, 
we still assume that $Y$ is a pure $d$-dimensional irreducible homogeneous variety
in $\C^n$ with an isolated singularity at the origin, and the desingularization $\pi: M\rightarrow Y$
is given by the blow up of the origin. In this situation, we remark
that the index $k_0(Y)=d-1$ in Theorem \ref{thm:l2}.\\

Now, let $V$ be an open neighborhood of $X$ in $M$, and assume that the $\dq$-equation
is solvable in the $L^\infty$-sense on $\Reg \pi(V)=\pi(V)^*$ as in Theorem \ref{thm:c0}.
Then it follows that
$$H^1(M,\mathcal{I}^1 \OO_M) \cong \bigoplus_{\mu\geq 1} H^1(X,\OO(N^{-\mu})) = 0.$$
This can be seen as follows: Let $[\omega]\in H^1(M,\mathcal{I}^1\OO_M)$
be represented by the $\dq$-closed $(0,1)$-form $\omega\in \mathcal{I}^1 \mathcal{C}^\infty_{0,1}(M)$.
It is easy to see that
$$\omega':=(\pi|_{V\setminus X}^{-1})^*\omega|_{V\setminus X} \in L^\infty_{0,1}(\pi(V)^*).$$
$\omega'$ is $\dq$-closed, and by assumption there exists
$\eta'\in L^\infty(\pi(V)^*)$
such that $\dq\eta'=\omega'$. So, we set
$$\eta:=\pi|_{V\setminus X}^* \eta' \in L^\infty(V\setminus X),$$
and it follows that
\begin{eqnarray}\label{eq:dq0}
\dq \eta = \omega \ \ \mbox{ on } V
\end{eqnarray}
by the Extension Theorem \ref{thm:extension} if we extend $\eta$ trivially across $X$.
But the right hand side of \eqref{eq:dq0} is smooth, and so we can assume $\eta\in C^\infty(V)$
after modifying it on a zero set. This implies that $\eta$ is a holomorphic function on the compact
complex manifold $X$, hence constant on $X$, and we can assume furthermore that it is vanishing on $X$.
But now, we deduce that $\eta\in \mathcal{I}^1 \mathcal{C}^\infty(V)$. The reason is as follows:
Let $Q\in X$. Then there exists a neighborhood $U_Q$ of $Q$ where we can find a solution
$f\in \mathcal{I}^1 \mathcal{C}^\infty(U_Q)$ such that $\dq f= \omega$ on $U_Q$.
So, $f-\eta$ is a holomorphic function on $U_Q$, vanishing on $X$, hence in $\mathcal{I}^1\OO(U_Q)$.
But then $\eta\in \mathcal{I}^1 \mathcal{C}^\infty(U_Q)$, too.
As in the proof of Theorem \ref{thm:cohom} (where we have used the results of Henkin and Leiterer),
this solution can be extended to the whole manifold $M$, and so:

\begin{thm}\label{thm:equivalent}
Let $Y$ be an irreducible homogeneous variety in $\C^n$ with an isolated singularity at the origin,
$X$ the exceptional set of the blow up at the origin, and $N$ the normal bundle of $X$ in this desingulaization.
If $\Omega\subset\subset Y$ is a strongly pseudoconvex domain with $0\in\Omega$ and $\Omega^*=\Omega\setminus\{0\}$,
then there exists a bounded linear $\dq$-solution operator
$${\bf S}_1: L^\infty_{0,1}(\Omega^*)\cap \ker\dq \rightarrow L^\infty(\Omega^*)$$
exactly if
$$H^1(X, \OO(N^{-\mu}))=0\ \ \mbox{ for all }\ \ \mu\geq 1.$$
\end{thm}

It is interesting to combine such results with Scheja's extension Theorem \ref{thm:scheja} 
in order to obtain statements about the exceptional set.\\

In the $L^2$ situation, we have seen in Theorem \ref{thm:l2} that vanishing of 
the groups 
$$H^1(X, \OO(N^{-\mu}))\ \ \mbox{ for all }\ \mu\geq 1-d$$
is a sufficient condition
for $L^2$-solvability of the $\dq$-equation (where $d=\dim Y$) on such spaces.

By using the method above, we would obtain 
$$H^1(X, \OO(N^{-\mu}))=0\ \ \mbox{ for all }\ \mu\geq 2-d$$
as a necessary condition,
and end with a gap at $H^1(X,\OO(N^{1-d}))$.
So, we forgo working out that statement in detail.

\section{Examples of Cones with Isolated Singularities}

We will now discuss some simple examples which allow the application
of Theorem \ref{thm:c0}, Theorem \ref{thm:hoelder} and Theorem \ref{thm:equivalent}.
Let $X'$ be a compact Riemannian Surface in $\C\mathbb{P}^{n-1}$
which is biholomorphically equivalent to $\C\mathbb{P}^1$. 
Let $Y$ be the cone in $\C^n$ associated to $X'$ as in the previous section,
and $X\cong X'\cong \C\mathbb{P}^1$
the exceptional set of the blow up at the origin.
Let $z_0\in X$ be an arbitrary point 
and $D= - (z_0)$ the associated divisor.
Then it follows that
\begin{eqnarray*}
H^j(X,\OO(D^\mu))\cong H^j(X,\OO(N^\mu))
\end{eqnarray*}
for all $j\geq 0$ and $\mu\in\Z$, where $N$ is the normal bundle of $X$ in the desingularization $\pi: M\rightarrow Y$.
It is well-known (and easy to calculate by power series)
that
$$\dim H^0(X,\OO(D^\mu)) = 1 - \mu\ \ \mbox{ for } \ \mu\leq 1.$$
Hence, we calculate by the Theorem of Riemann-Roch that
\begin{eqnarray*}
- \dim H^1(X,\OO(N^\mu)) = \deg D^\mu + 1 - \mbox{genus}(\C\mathbb{P}^1) - (1-\mu) = 0
\end{eqnarray*}
for all $\mu\leq 1$, which guarantees $L^2$ and $L^\infty$-solvability of the $\dq$-equation on $Y$, because it follows that
$$H^1(M,\mathcal{I}^{-1}\OO_M) \cong \bigoplus_{\mu\geq -1} H^1(X,\OO(N^{-\mu})) =0.$$
An important example for such a variety is $Y=\{(x,y,z)\in\C^3: xy=z^2\}$.

\vspace{2mm}
As a second example, we use the same construction but assume that $X\cong X'$ is an elliptic curve.
Here, $H^0(X,\OO(D^\mu))$ is the space of elliptic functions with a single pole of order $-\mu$.
So, it is well-known that we have
$$\dim H^0(X,\OO(D^\mu)) =\left\{
\begin{array}{ll}
1 &, \mbox{ for } \mu = 0,\\
-\mu &, \mbox{ for } \mu\leq -1.
\end{array}\right.$$
Using the Theorem of Riemann-Roch again, we calculate
\begin{eqnarray*}
- \dim H^1(X,\OO(N^\mu)) = \deg D^\mu + 1 - \mbox{genus}(X) - (-\mu)
= -\mu + 1 - 1 +\mu=0
\end{eqnarray*}
for all $\mu\leq -1$, and
$L^\infty$-solvability of the $\dq$-equation on $Y$ follows from
$$H^1(M,\mathcal{I}^1\OO_M) \cong \bigoplus_{\mu\geq 1} H^1(X,\OO(N^{-\mu})) =0.$$
Examples are the varieties $Y=\{(x,y,z)\in\C^3: y^2z=x^3+axz^2 + bz^3\}$
for suitable values of $a$ and $b$.

\newpage
\section{A Remark on the Cohomology of the Desingularization}

It is worth mentioning explicitly that we have proven 
$$H^q(M,\OO_M) \cong \bigoplus_{\mu\geq 0} H^q(X,\OO(N^{-\mu}))\ \ \mbox{ for } q\geq 1,$$
where $N$ is the normal bundle of $X$ in $M$. That can be seen as follows:\\

Theorem \ref{thm:embedding} implies
that there is a neighborhood $V$ of the zero section in $N$
which is biholomorphically equivalent to a neighborhood $\Psi(V)$ of $X$ in $M$.
Consider $\phi: M\rightarrow \R$ given by $\phi(w):=\|\pi(w)\|^2$.
Then $\phi$ is a smooth exhaustion function on $M$ which is strictly plurisubharmonic on $M\setminus X$,
and there exists $c_0>0$ such that 
$$D_0:=\{w\in M: \phi(w)<c_0\}\subset\subset \Psi(V)$$
has smooth boundary. Then, by \cite{Hoe}, Theorem 3.4.9, 
$$H^q(D_0,\OO_M) \cong H^q(M,\OO_M)\ \ ,\ q\geq 1.$$
But on the other hand
$$D_1:= \Psi^{-1}(D_0) \subset\subset V \subset\subset N$$
is a strongly pseudoconvex neighborhood of the zero section in the negative line bundle $N$,
and therefore Theorem \ref{thm:cohom} yields:
$$H^q(D_0,\OO_M) \cong H^q(D_1,\OO_N) \cong \bigoplus_{\mu\geq 0} H^q(X,\OO(N^{-\mu}))\ \ \mbox{ for } q\geq 1.$$
Summing up, we get:

\begin{thm}\label{thm:cohomM}
Let $Y$ be a pure dimensional analytic set in $\C^n$
with an isolated singularity at the origin
such that the exceptional set $X$ of a desingularization
$\pi: M\rightarrow Y$ is regular. Then:
$$H^q(M,\OO_M) \cong \bigoplus_{\mu\geq 0} H^q(X,\OO(N^{-\mu}))\ \ \mbox{ for } q\geq 1,$$
where $N$ is the normal bundle of $X$ in $M$.
\end{thm}


\end{document}